\documentclass[12pt]{amsart}

\usepackage{CJKutf8}

\usepackage{amssymb}
\usepackage{amsmath}
  \usepackage{paralist}
\usepackage{enumitem}
\usepackage{comment}

 \usepackage[colorlinks=true]{hyperref}
\hypersetup{urlcolor=blue, citecolor=red}
\usepackage{hyperref}

\newtheorem{thmm}{Theorem}

\newtheorem*{TheoremD'}{Theorem D'}

\newtheorem*{TheoremE'}{Theorem E'}

\newtheorem{theorem}{Theorem}[section]
\newtheorem{corollary}[theorem]{Corollary}

\newtheorem*{main*}{Main Theorem}
\newtheorem{lemma}[theorem]{Lemma}
\newtheorem{proposition}[theorem]{Proposition}

\theoremstyle{definition}
\newtheorem{definition}[theorem]{Definition}
\newtheorem{remark}[theorem]{Remark}

\newcommand{\F}{\mathcal{F}}

\def\pX{{\partial X}}
\def\RR{{\mathbb R}}
\def\NN{{\mathbb N}}

\def\ZZ{{\mathbb Z}}

\def\inj{{\text{inj}}}
\def\lv{{\underline v}}
\def\lc{{\underline c}}
\def\lw{{\underline w}}
\def\lB{{\underline B}}
\def\lS{{\underline S}}
\def\lm{{\underline m}}
\def\pr{{text{pr}}}

\def\diam{\mathop{\hbox{{\rm diam}}}}

\def\diam{\mathop{\hbox{{\rm diam}}}}
\def\pr{\mathop{\hbox{{\rm pr}}}}

\def\loc{{\mathop{\hbox{\footnotesize  \rm loc}}}}

\def\top{{\mathop{\hbox{\footnotesize \rm top}}}}

\def\a{\alpha}
\def\b{\beta}
\def\c{\gamma}   \def\C{\Gamma}
\def\d{\delta}   
   
 \def\e{\epsilon}

\def\ae{\text{-a.e.}\ }
\def\bP{\textbf{P}}
\def\bF{\textbf{F}}

\title[Running heading with forty characters or less]
      {Counting closed geodesics on rank one manifolds without focal points}

\author[first-name1 last-name1 and first-name2 last-name2]{Weisheng Wu}


\address{School of Mathematical Sciences, Xiamen University, Xiamen, 361005, P.R. China}
\email{wuweisheng@math.pku.edu.cn}

\begin{document}

\maketitle
\markboth{Counting closed geodesics on rank one manifolds without focal points}
{W. Wu}
\renewcommand{\sectionmark}[1]{}

\begin{abstract}
In this article, we consider a closed rank one Riemannian manifold $M$ without focal points. Let $P(t)$ be the set of free-homotopy classes containing a closed geodesic on $M$ with length at most $t$, and $\# P(t)$ its cardinality. We obtain the following Margulis-type asymptotic estimates:
\[\lim_{t\to \infty}\#P(t)/\frac{e^{ht}}{ht}=1\]
where $h$ is the topological entropy of the geodesic flow. In the appendix, we also show that the unique measure of maximal entropy of the geodesic flow has the Bernoulli property.
\end{abstract}

\section{Introduction}
Consider a closed Riemannian manifold $(M,g)$. When $M$ has negative sectional curvature everywhere, the geodesic flow defined on the unit tangent bundle $SM$ is a prime example of Anosov flow (cf. \cite{An}, \cite[Section 17.6]{KH}). There is correspondence between closed geodesics on $M$ and periodic orbits under the geodesic flow, and they are important from both geometric and dynamical points of view. The asymptotic growth of the number of closed geodesics is given by the celebrated theorem of Margulis in his 1969 thesis \cite{Mar1, Mar2}:
\begin{equation}\label{e:mar}
\begin{aligned}
\lim_{t\to \infty}\#P(t)/\frac{e^{ht}}{ht}=1
\end{aligned}
\end{equation}
where $P(t)$ is the set of free-homotopy classes containing a closed geodesic with length at most $t$, and $h$ is the topological entropy of the geodesic flow.

The main ingredient in the proof of Margulis's theorem is from ergodic theory. Margulis \cite{Mar2} constructed an invariant measure under the geodesic flow. This measure is mixing, and the conditional measures on stable/unstable manifolds contract/expand with a uniform rate under the flow. Using different methods, Bowen constructed a measure of maximal entropy (MME for short) in \cite{Bo1}, for Axiom A flows. Later in \cite{Bo2}, Bowen proved that MME is unique. It turns out that the measures constructed by Bowen and Margulis are eventually the same one, which is called \emph{Bowen-Margulis measure}. Bowen-Margulis measure has the Bernoulli property \cite{Rat}, an ultimate mixing property for measurable dynamical systems.

An alternative approach to \eqref{e:mar} for Axiom A flows is given by Parry and Pollicott \cite{PP}, based on zeta functions and symbolic dynamics. See the survey by Sharp in \cite{Mar2} for further progress that has been made since Margulis's thesis.

In 1984, A.~Katok \cite{BuKa} conjectured that the geodesic flow on a closed rank one manifold of nonpositive curvature admits a unique MME. On such manifolds, the geodesic flow exhibits nonuniformly hyperbolic behavior. In 1998, Katok's conjecture was finally settled by Knieper \cite{Kn2}. In his proof, Knieper used Patterson-Sullivan measures on the boundary at infinity of the universal cover of $M$ to construct a MME (called \emph{Knieper measure}), and showed that this measure is the unique MME. Recently, Burns, Fisher, Climenhaga and Thompson \cite{BCFT} proved the uniqueness of equilibrium states for certain potentials, which generalized Knieper's result.

Knieper's measure can be used to obtain the following asymptotic estimates \cite{Kn1, Kn3}: there exists $C>0$ such that for $t$ large enough
$$\frac{1}{C}\frac{e^{ht}}{t}\le \#P(t)\le C \frac{e^{ht}}{t}.$$
Since then, efforts are made to improve the above to the Margulis-type asymptotic estimates \eqref{e:mar}. A preprint \cite{Gun}, though not published, contains many inspiring ideas to this problem. A breakthrough was made recently by Ricks \cite{Ri}: \eqref{e:mar} is established for rank one locally CAT$(0)$ spaces, which include rank one manifolds with nonpositve curvature.

Beyond nonpositive curvature, manifolds without focal/conjugate points are studied extensively. Following Knieper's method, Liu, Wang and the author construct the Knieper measure for manifolds without focal points and showed it is
the unique MME \cite{LWW}. Gelfert and Ruggiero \cite{GR} also proved the uniqueness of MME for surfaces without focal points
by different methods. Extending the approach in \cite{BCFT}, Chen, Kao and Park proved the uniqueness of equilibrium states for certain potentials in \cite{CKP1, CKP2}. Manifolds without conjugate points in general are far from being well understood. Remarkable progress was made by Climenhaga, Knieper and War for a class of manifolds (including all surfaces) without conjugate points: the uniqueness of MME is proved in \cite{CKW1} and \eqref{e:mar} is obtained in \cite{CKW2}. Weaver \cite{We} also obtained \eqref{e:mar} for surfaces having negative curvature outside of a collection of radially symmetric caps.

In this paper, we obtain Margulis-type asymptotic estimates \eqref{e:mar} for manifolds without focal points, by combining ideas from \cite{Ri} and \cite{CKW2}. Compared to rank one locally CAT$(0)$ or manifolds of nonpositive curvature in \cite{Ri}, the lack of convexity properties causes some difficulties, and some rigidity and comparison results do not hold any more (see Remarks \ref{lack}, \ref{fact1} and \ref{fact3} below). We will provide full account of details on the angle metric, weak $\pi$-convergence theorem, continuity of Busemann functions in the setting of no focal points. We also adopt the new ideas in \cite{CKW2}: for example, the slice instead of full box is used in the intersection $S\cap \phi^{-t} \c B$ (see Section 2.4 for notations).

As we are dealing with manifolds, we will follow more closely the terminology in \cite{CKW2} than \cite{Ri}. Nevertheless, our setting is quite different from that in \cite{CKW2}. For example, manifolds without focal points do not have uniform visibility property in general, and thus we have to develop weak $\pi$-convergence theorem to prove a type of closing lemma. Moreover in \cite{CKW2}, by Morse Lemma the action induced by a deck transformation on the boundary at infinity has a unique repelling/attracting fixed point (\cite[Remark 2.6]{CKW2}). In our setting, the action has no such property and its behaviour will be analyzed using the properties of no focal points.

In the appendix, we show that the Knieper measure for manifolds without focal points has the Bernoulli property. In nonpositive curvature case, the Knieper measure is proved to be mixing in \cite{Ba} and Bernoulli in \cite{CT}. For manifolds without focal points, Knieper measure is known to be mixing by \cite{LLW} and have the Kolmogorov property by \cite{CKP2}. The method of Call and Thompson \cite{CT}, actually proves the Kolmogorov property of the unique equilibrium states for certain potentials \cite{CT,CKP2}. For Knieper measure, to lift the Kolmogorov property to the Bernoulli property, we adapt the classical argument in \cite{OW1, Rat} for Anosov flows. In fact, we follow closely \cite{CH}, which is for systems with nonuniformly hyperbolic structure.

\subsection{Statement of main results}
Suppose that $(M,g)$ is a $C^{\infty}$ closed $n$-dimensional Riemannian manifold,
where $g$ is a Riemannian metric. Let $\pi: SM\to M$ be the unit tangent bundle over $M$. For each $v\in S_pM$,
we always denote by $c_{v}: \RR\to M$ be the unique geodesic on $M$ satisfying the initial conditions $c_v(0)=p$ and $\dot c_v(0)=v$. The geodesic flow $\phi=(\phi^{t})_{t\in\mathbb{R}}$ (generated by the Riemannian metric $g$) on $SM$ is defined as:
\[
\phi^{t}: SM \rightarrow SM, \qquad (p,v) \mapsto
(c_{v}(t),\dot c_{v}(t)),\ \ \ \ \forall\ t\in \RR .
\]

A vector field $J(t)$ along a geodesic $c:\RR\to M$ is called a \emph{Jacobi field} if it satisfies the \emph{Jacobi equation}:
\[J''+R(J, \dot c)\dot c=0\]
where $R$ is the Riemannian curvature tensor and\ $'$\ denotes the covariant derivative along $c$.
\begin{definition}
Let $c$ be a geodesic on $(M,g)$.
\begin{enumerate}
  \item A pair of distinct points $p=c(t_{1})$ and $q=c(t_{2})$ are called \emph{focal} if there is a Jacobi field $J$ along $c$ such that $J(t_{1})=0$, $J'(t_{1})\neq 0$ and $\frac{d}{dt}\| J(t)\|^{2}\mid_{t=t_{2}}=0$;
  \item $p=c(t_{1})$ and $q=c(t_{2})$ are called \emph{conjugate} if there is a nontrivial Jacobi field $J$ along $c$ such that $J(t_{1})=0=J(t_{2})$.
\end{enumerate}
A compact Riemannian manifold $(M,g)$ is called a manifold \emph{without focal points/without conjugate points} if there is no focal points/conjugate points on any geodesic in $(M,g)$.
\end{definition}

By definition, if a manifold has no focal points then it has no conjugate points. All manifolds of nonpositive curvature always have no focal points.

A Jocobi field $J(t)$ along a geodesic $c(t)$ is called \emph{parallel} if $J'(t)=0$ for all $t\in \RR$. The notion of \emph{rank} is defined as follows.
\begin{definition}
For each $v \in SM$, we define \text{rank}($v$) to be the dimension of the vector space of parallel Jacobi fields along the geodesic $c_{v}$, and \text{rank}($M$):=$\min\{$\text{rank}$(v) \mid v \in SM\}$. For a geodesic $c$ we define \text{rank}($c$)=\text{rank}($\dot c(t)$), $\forall\ t\in \mathbb{R}$.
\end{definition}

In this paper, we always let $M$ be a rank one closed Riemannian manifold without focal points. Then $SM$ splits into two invariant subsets under the geodesic flow: the regular set $\text{Reg}:= \{v\in SM \mid \text{rank}(v)=1\}$, and the singular set $\text{Sing}:= SM \setminus \text{Reg}$.

Recall that $P(t)$ is the set of free-homotopy classes containing a closed geodesic with length at most $t$. If $M$ has negative curvature, each free-homotopy class contains exactly one closed geodesic, so $\#P(t)$ is just the number of closed geodesics with length at most $t$. If $M$ has nonpositive curvature or without focal points, a free-homotopy class may have infinitely many closed geodesics.

The main result of this paper is the following Margulis-type asymptotic estimates of $\#P(t)$.
\begin{thmm}\label{margulis}
Let $M$ be a rank one closed Riemannian manifold without focal points, then
$$\#P(t)\sim\frac{e^{ht}}{ht}$$
where $h=h_{\text{top}}(g)$ denotes the topological entropy of the geodesic flow on $SM$, and the notation $f_1\sim f_2$ means that $\lim_{t\to\infty}\frac{f_1(t)}{f_2(t)}=1$.
\end{thmm}

Recall that a $\phi$-invariant probability measure $\mu$ is called the \emph{measure of maximal entropy} (MME for short) if $h_{\mu}(\phi)\geq h_{\nu}(\phi)$ for any $\phi$-invariant probability measure $\nu$. By the variational principle, $h_{\mu}(\phi)=h$. It is shown in \cite{LLW} that MME is unique and can be constructed via Patterson-Sullivan measures. The unique MME is called \emph{Knieper measure} (see Section 2.3 for more details).

The following is a result of equidistribution of periodic orbits with respect to the Knieper measure.

\begin{thmm}\label{equi}
Suppose that $M$ is a rank one closed Riemannian manifold without focal points, and $\e\in (0, \inj (M)/2$ is fixed where $\inj (M)$ is the injectivity radius of $M$.
For $t > 0$, let $C(t)$ be any maximal set of pairwise non-free-homotopic
closed geodesics with lengths in $(t-\e, t]$, and define the measure
$$\nu_t:=\frac{1}{C(t)}\sum_{c\in C(t)}\frac{Leb_c}{t}$$
where $Leb_c$ is the Lebesgue measure along the curve $\dot c$ in the unit
tangent bundle $SM$.

Then the measures $\nu_t$ converge in the
weak$^*$ topology to the unique measure of maximal entropy as $t\to \infty$.
\end{thmm}

In Section $2$, we recall some geometric and ergodic background of manifolds without foal points, including Busemann functions, construction of the Knieper measure and local product flow boxes. All Sections $3-6$ are devoted to proving Theorem \ref{margulis}. Meanwhile, Theorem \ref{equi} is proved in Section $6$.

In the appendix, we prove the following ultimate mixing property for the Knieper measure. We also recall the necessary definitions of Kolmogorov/Bernoulli properties and known results from the literature.
\begin{thmm}\label{bernoulli}
Let $M$ be a rank one closed Riemannian manifold without focal points, then
the unique measure of maximal entropy is Bernoulli.
\end{thmm}

\section{Geometric and ergodic background}
In this section, we present some geometric and ergodic results on rank one manifolds without focal points, which will be used in the following discussions.

\subsection{Geometric properties and boundary at infinity}
Let $M$ be a closed Riemannian manifold without focal points, and $\pr: X\to M$ the universal cover of $M$. Let $\C \simeq \pi_1(M)$ be the group of deck transformations on $X$, so that each $\c\in \C$ acts isometrically on $X$. Since $M=X/\Gamma$ is compact, each $\c\in \Gamma$ is \emph{axial} (cf. \cite{CS} Lemma 2.1), that is, there exists a geodesic $c$ and $t_{0}>0$ such that $\c(c(t)) = c(t+t_{0})$ for every $t\in \mathbb{R}$. Correspondingly $c$ is called an \emph{axis} of $\c$ and we denote $|\c|:=t_0$ where $t_0$ is minimal with the above property.

We still denote by $\pr: SX\to SM$ and $\c: SX\to SX$ the map on unit tangent bundles induced by $\pr$ and $\c\in \C$. From now on, we use an underline to denote objects in $M$ and $SM$, e.g. for a geodesic $c$ in $X$ and $v\in SX$, $\lc:=\pr c$, $\lv:=\pr v$ denote their projections to $M$ and $SM$ respectively.
Denote by $d$ both the distance functions on $M$ and $X$ induced by Riemannian metrics. The Sasaki metrics on $SM$ and $SX$ are also denoted by $d$ if there is no confusion.

Suppose that $c_{1}$ and $c_{2}$ are both geodesics in $X$. We call $c_{1}$ and $c_{2}$ are \emph{positively asymptotic} or just \emph{asymptotic} if there is a positive number $C > 0$ such that
\begin{equation}\label{e1}
d(c_{1}(t),c_{2}(t)) \leq C, ~~\forall~ t \geq 0.
\end{equation}
We say $c_{1}$ and $c_{2}$ are \emph{negatively asymptotic} if \eqref{e1} holds for all $t \leq 0$. $c_{1}$ and $c_{2}$ are said to be \emph{bi-asymptotic} or \emph{parallel} if they are both positively and negatively asymptotic.
The relation of positive/negative asymptoticity is an equivalence relation between geodesics on $X$. The class of geodesics that are positively/negatively asymptotic to a given geodesic $c_v$ is denoted by $c_v(+\infty)$/$c_v(-\infty)$ or $v^+/v^-$ respectively. We call them \emph{points at infinity}. Obviously, $c_{v}(-\infty)=c_{-v}(+\infty)$. We use $\pX$ to denote the set of all points at infinity,
and call it the \emph{boundary at infinity}, or the \emph{ideal boundary}. If $\eta=v^+\in \pX$, we say $v$ \emph{points at $\eta.$}

The following are the fundamental properties of the geometry of manifolds without focal points.
\begin{lemma}(Cf. \cite{OS})\label{nonincreasing}
Let $X$ be a simply connected closed Riemannian manifold without focal points.
\begin{enumerate}
  \item Let $c_1$ and $c_2$ be distinct geodesics with $c_1(0)=c_2(0)$.
Then for $t>0$, both $d(c_1(t), c_2)$ and $d(c_1(t), c_2(t))$ are strictly increasing and tend to infinity as $t\to \infty$.
  \item Let $c_1$ and $c_2$ be asymptotic geodesics. Then both $d(c_1(t), c_2)$
and $d(c_1(t), c_2(t))$ are nonincreasing functions of $t\in \RR$.
\item For any geodesic $c$ and each $p\in X$, there exists a unique geodesic through $p$ and asymptotic to $c$.
\item (Flat Strip Lemma)  If $c_1$ and $c_2$ are bi-asymptotic, then they bound a flat strip, i.e., there is an isometric embedding $\phi: [0,a]\times \RR\to X$ such that $\phi(0,t)=c_1(t)$ and $\phi(a,t)=c_2(t)$ up to parametrization.
\end{enumerate}
\end{lemma}
\begin{remark}\label{lack}
If $X$ has nonpositive curvature, both the functions $d(c_1(t), c_2)$ and $d(c_1(t), c_2(t))$ are convex in $t\in \RR$ for any two geodesics $c_1$ and $c_2$. This is not true for general manifolds without focal points. The lack of convexity causes difficulties in many problems for manifolds without focal points.
\end{remark}

We can define the visual topology on $\pX$ following \cite{Eb1} and \cite{EO}. For each $p$, by Lemma \ref{nonincreasing}, there is a bijection $f_p: S_pX\to \pX$ defined by $f_p(v)=v^+, v\in S_pX$. So for each $p\in M$, $f_p$ induces a topology on $\pX$ from the usual topology on $S_pX$. Given $p,q\in X$, let $\psi: S_pX\to S_qX$ be the map such that $\psi(v)$ is the unique vector in $S_qX$ asymptotic to $v\in S_pX$.
\begin{lemma}(\cite[Lemma 2.6]{Wat})
The map $\psi: S_pX\to S_qX$ above is continuous.
\end{lemma}
It follows that $\psi$ is in fact a homoeomorphism. Hence the topology on $\pX$ induced by $f_p$ is independent of $p\in X$, and is called the \emph{visual topology} on $\pX$.

The above topology on $\pX$ and the manifold topology on $X$ can be extended to $\overline X:= X\cup \pX$ naturally by requiring the map $\varphi$ defined as follows is a homeomorphism. Fix $p\in X$. For each $v\in T_pX$ with $\|v\|\le 1$, define
\[
\varphi(v):=\begin{cases}\exp (\frac{v}{1-\|v\|}) &
   \text{if} \ \|v\|<1; \\
f_p(v) &\text{if} \ \|v\|=1.
\end{cases}
\]
This topology is usually called the \emph{cone topology}. Under this topology, $\overline{X}$ is homeomorphic to the closed unit ball in $\mathbb{R}^{\text{dim}(X)}$, and $\pX$ is homeomorphic to the unit sphere $\mathbb{S}^{\text{dim}(X)-1}$.

The following continuity property is useful.
\begin{lemma}\label{con}(\cite[Lemma 6.4]{Wat})
Let $p, p_n\in X$ with $p_n\to p$, and $x_n,\zeta\in \overline X$ with $x_n\to \zeta$. Then $\dot c_{p_nx_n}(0)\to \dot c_{p\zeta}(0)$, where $c_{qx}$ is the unique geodesic from $q\in X$ and pointing at $x\in \overline X$.
\end{lemma}
\subsection{Continuity of Busemann function}
For each pair of points $(p,q)\in X \times X$ and each point at infinity $\xi \in \pX$, the \emph{Busemann function based at $\xi$ and normalized by $p$} is
$$b_{\xi}(q,p):=\lim_{t\rightarrow +\infty}\big(d(q,c_{p,\xi}(t))-t\big),$$
where $c_{p,\xi}$ is the unique geodesic from $p$ and pointing at $\xi$.
The Busemann function $b_{\xi}(q,p)$ is well-defined since the function $t \mapsto d(q,c_{p,\xi}(t))-t$
is bounded from above by $d(p,q)$, and decreasing in $t$ (this can be checked by using the triangle inequality). If $v\in S_pX$ points at $\xi\in \pX$, we also write $b_v(q):=b_{\xi}(q,p).$

The level sets of the Busemann function $b_{\xi}(q,p)$ are called the \emph{horospheres} centered at $\xi$. The horosphere through $p$ based at $\xi\in \pX$, is denoted by $H_p(\xi)$.
For more details of the Busemann functions and horospheres, please see ~\cite{DPS,Ru1,Ru2}.
Here we are concerned with the continuity property of the horospheres and Busemann functions.

We say that the manifold $M$ satisfies the \emph{axiom of asymptoticity} if for every $v\in SX$
the following statement holds:
For any choice of $x_n, x\in X, v_n\in SX$, $x_n\to x$, $v_n\to v$ and a sequence of numbers $t_n\to \infty$, the sequence $c_n$ of geodesic segments joining $x_n$ to $c_{v_n}(t_n)$ converges to an asymptote of $c_v$.

\begin{proposition}(Cf. \cite[Theorem 6.1]{Pe2} \cite[Lemma 1.2]{Ru1})\label{horosphere}
Let $M$ be a closed Riemannian manifold without conjugate points that satisfies the the axiom of asymptoticity. Then for every $p\in X$, the map $\xi\mapsto H_\xi(p)$ is continuous in the following sense: if $\xi_n\to \xi \in \pX$ and $K\subset X$ is compact, then $H_{\xi_n}(p)\cap K\to H_\xi(p)\cap K$ uniformly in the Hausdorff topology.
\end{proposition}

\begin{proposition}(Cf. \cite[Theorem 5.2]{Pe2})\label{axiom}
If the manifold $M$ has no focal points, then it satisfies the the axiom of asymptoticity.
\end{proposition}

By Proposition \ref{axiom}, Proposition \ref{horosphere} applies to manifolds without focal points. So we have the following corollary (see also \cite[Proposition 9]{LWW}):
\begin{corollary}\label{continuous}
The functions $(v,q)\mapsto b_v(q)$ and $(\xi,p,q)\mapsto b_\xi(p,q)$ are continuous on $SX\times X$ and $\pX\times X\times X$ respectively.
\end{corollary}

We would like to obtain certain equicontinuity of Busemann function $v\mapsto b_v(q)$. The idea is to extend the function to the compact space $X\cup \pX$.
Given $x,p,q\in X$, define
$$b_x(q,p):=d(q,x)-d(p,x).$$

\begin{lemma}\label{continuity1}(\cite[Proposition 10]{LWW})
For each pair of points $p,q\in X$, if there is a sequence $\{x_{n}\}\subset X$ with $\lim_{n\rightarrow +\infty}x_{n}=\xi\in \pX$, then $$\lim_{n\rightarrow +\infty}b_{x_n}(q,p)=b_{\xi}(q,p).$$
\end{lemma}

\begin{lemma}\label{equicon1}
Let $p\in X$, $A\subset S_pX$ be closed, and $B\subset X$ be such
that $A^+:= \{v^+: v\in A\}$ and $B^\infty := \{\lim_n q_n \in \pX: q_n\in B\}$ are
disjoint subsets of $\pX$. Then the family of functions $A\to \RR$ indexed
by $B$ and given by $v\mapsto b_v(q)$ are equicontinuous: for every $\e>0$ there
exists $\d >0$ such that if $\angle_p(v,w)<\d$, then $|b_v(q)-b_w(q)|<\e$ for every $q\in B$.
\end{lemma}
\begin{proof}
The proof is a repetition of that of \cite[Lemma A.1]{CKW2} and thus omitted. Notice that in our setting, Lemma \ref{continuity1} should be used instead of \cite[Corollary 2.18]{CKW1}.
\end{proof}

\subsection{Patterson-Sullivan measure and Knieper measure}
We recall the construction of Patterson-Sullivan measure and Knieper measure in this section. Let us start with the definition of  Busemann density.
\begin{definition}\label{density}
Let $X$ be a simply connected manifold without focal points and
$\Gamma$ a discrete subgroup of $\text{Iso}(X)$, the group of isometries of $X$. For a given constant $r>0$, a family of finite Borel measures $\{\mu_{p}\}_{p\in X}$
on $\pX$ is called an $r$-dimensional \emph{Busemann density} if
\begin{enumerate}
\item For any $p,\ q \in X$ and $\mu_{p}$-a.e. $\xi\in \pX$,
$$\frac{d\mu_{q}}{d \mu_{p}}(\xi)=e^{-r \cdot b_{\xi}(q,p)}$$
 where $b_{\xi}(q,p)$ is the Busemann function.
\item $\{\mu_{p}\}_{p\in X}$ is $\Gamma$-equivariant, i.e., for all Borel sets $A \subset \pX$ and for any $\c \in \Gamma$,
we have $$\mu_{\c p}(\c A) = \mu_{p}(A).$$
\end{enumerate}
\end{definition}
The construction of such a Busemann density is due to Patterson \cite{Pat} in the case of Fuchsian groups, and generalized by Sullivan \cite{Sul} for hyperbolic spaces.

Extending the techniques in \cite{Kn2} to manifolds without focal points, we constructed Busemann density via Patterson-Sullivan construction and showed the following in \cite{LWW}.
\begin{theorem}\label{unique1}(\cite[Theorem B]{LWW})
Let $M=X/\Gamma$ be a compact rank $1$ manifold without focal points, then up to a multiplicative constant, the Busemann density is unique, i.e., the Patterson-Sullivan measure is the unique Busemann density.
\end{theorem}

Let $\{\mu_{p}\}_{p\in X}$  be the Patterson-Sullivan measure, which is the unique Busemann density with dimension $h=h_{\top}(g)$.
Let $P:SX \to \pX \times \pX$ be the projection given by $P(v)=(v^-,v^+)$. Denote by $\mathcal{I}^{P}:=P(SX)=\{P(v)\mid v \in SX\}$ the subset of pairs in $\pX \times \pX$
which can be connected by a geodesic. Note that the connecting geodesic may not be unique and moreover, not every pair $\xi\neq \eta$ in $\pX$ can be connected by a geodesic.

Fix a point $p\in X$, we can define a $\Gamma$-invariant measure $\overline{\mu}$
on $\mathcal{I}^{P}$ by the following formula:
\begin{equation}\label{mme1}
d \overline{\mu}(\xi,\eta) := e^{h\cdot \beta_{p}(\xi,\eta)}d\mu_{p}(\xi) d\mu_{p}(\eta),
\end{equation}
where $\beta_{p}(\xi,\eta):=-\{b_{\xi}(q, p)+b_{\eta}(q, p)\}$ is the Gromov product, and $q$ is any point on a geodesic $c$ connecting $\xi$ and $\eta$. It's easy to see that the function $\beta_{p}(\xi,\eta)$ does not depend on the choice of $c$ and $q$.
In geometric language, the Gromov product $\beta_{p}(\xi,\eta)$ is the length of the part of a geodesic
$c$ between the horospheres $H_{\xi}(p)$ and $H_{\eta}(p)$.

Then $\overline{\mu}$ induces a $\phi$-invariant measure $m$ on $SX$ with
\begin{equation}\label{mme2}
m(A):=\int_{\mathcal{I}^{P}} \text{Vol}\{\pi(P^{-1}(\xi,\eta)\cap A)\}d \overline{\mu}(\xi,\eta),
\end{equation}
for all Borel sets $A\subset SX$. Here $\pi : SX \rightarrow X$ is the standard projection map and
Vol is the induced volume form on $\pi(P^{-1}(\xi,\eta))$.
If there are more than one geodesics connecting $\xi$ and $\eta$,
and one of them has rank $k\geq 1$, then by the flat strip theorem, all of these connecting geodesics have rank $k$.
Thus, $\pi(P^{-1}(\xi,\eta))$ is exactly the $k$-flat submanifold connecting $\xi$ and $\eta$,
which consists of all the rank $k$ geodesics connecting $\xi$ and $\eta$.

For any Borel set $A\subset SX$ and $t\in \mathbb{R}$,
$\text{Vol}\{\pi(P^{-1}(\xi,\eta)\cap \phi^{t}A)\}=\text{Vol}\{\pi(P^{-1}(\xi,\eta)\cap A)\}$.
Therefore $m$ is $\phi$-invariant. Moreover, $\Gamma$-invariance of $\overline{\mu}$ leads to the $\Gamma$-invariance of $m$.
Thus $m$ induced a $\phi$-invariant measure on $SM$ by
\begin{equation*}
m(A)=\int_{SM} \#(\text{pr}^{-1}(\lv)\cap A)d\lm(\lv).
\end{equation*}
It is proved in \cite{LWW} that $\lm$ is unique MME, which is called Knieper measure.
Furthermore, $\lm$ is proved to be mixing in \cite{LLW} and Kolmogorov in \cite{CKP2}. In the appendix, we prove that $\lm$ is Bernoulli.

\subsection{Local Product flow boxes}
From now on we fix $v_0\in SX$ such that $\lv_0:=\pr v_0$ is a periodic and regular vector in $SM$. Let $p:=\pi(v_0)$, which will be the reference point in the following discussions. For simplicity, we will suppress $v_0$ and $p$ from the notation. We also fix a scale $\e\in (0, \min\{\frac{1}{8}, \frac{\inj (M)}{4}\})$.

The \emph{Hopf map} $H: SX\to \pX\times \pX\times \RR$ for $p\in X$ is defined as
$$H(v):=(v^-, v^+, s(v)), \text{\ where\ }s(v):=b_{v^-}(\pi v, p).$$
From definition, we see
$s(\phi^t v)=s(v)+t$
for any $v\in SX$ and $t\in \RR$. $s$ is continuous by Corollary \ref{continuous}.

Following \cite{CKW2}, we define for each $\theta>0$ and $0<\a<\frac{3}{2}\e$,
\begin{equation*}
\begin{aligned}
&\bP=\bP_\theta:=\{w^-: w\in S_pX \text{\ and\ }\angle_p(w, v_0)\le \theta\},\\
&\bF=\bF_\theta:=\{w^+: w\in S_pX \text{\ and\ }\angle_p(w, v_0)\le \theta\},\\
&B=B_\theta^\a:=H^{-1}(\bP\times \bF\times [0,\a]),\\
&S=S_\theta:=B_\theta^{\e^2}=H^{-1}(\bP\times \bF\times [0,\e^2]).
\end{aligned}
\end{equation*}
$B=B_\theta^\a$ is called a flow \emph{box} with depth $\a$, and $S=S_\theta$ is a \emph{slice} with depth $\e^2$. We will consider $\theta>0$ small enough, which will be specified in the following.

The following lemma is a variation of \cite[Lemma 1]{Ri}, which was crucial in constructing Knieper measure.
\begin{proposition}\label{crucial}(\cite[Propositions 6 and 7]{LWW})
Let $X$ be a simply connected manifold without focal points and $v_0\in SX$ is regular. Then for any $\e>0$, there is an $\theta_1 > 0$ such that, for any $\xi \in \bP_{\theta_1}$ and $\eta \in \bF_{\theta_1}$,
there is a unique geodesic $c_{\xi,\eta}$ connecting $\xi$ and $\eta$, i.e.,
$c_{\xi,\eta}(-\infty)=\xi$ and $c_{\xi,\eta}(+\infty)=\eta$.

Moreover, the geodesic $c_{\xi,\eta}$ is regular and $d(c_{v}(0), c_{\xi,\eta})<\epsilon/10$.
\end{proposition}

Based on Proposition \ref{crucial}, we have the following result.
\begin{lemma}\label{diameter}
Let $v_0, p, \e$ be as above and $\theta_1$ be given in Proposition \ref{crucial}. Then for any $0<\theta<\theta_1$,
\begin{enumerate}
  \item $\text{diam\ } \pi H^{-1}(\bP\times \bF\times \{0\})<\frac{\e}{2}$;
  \item $H^{-1}(\bP\times \bF\times \{0\})\subset SX$ is compact;
  \item $\text{diam\ } \pi B_\theta^\a <4\e$ for any $0<\a\le \frac{3\e}{2}$.
\end{enumerate}
\end{lemma}
\begin{proof}
Assume the contrary of (1), that is, there exist $\theta_n\to 0$, $v_n,w_n\in H^{-1}(\bP_{\theta_n}\times \bF_{\theta_n}\times \{0\})$ such that $d(\pi v_n,\pi w_n)\ge \frac{\e}{2}$.

By Proposition \ref{crucial}, for $n$ large enough, we know that $d(p, c_{v_n})<\frac{\e}{8}$ and $d(p, c_{w_n})<\frac{\e}{8}$. Let $v_n'=\phi^{t_n}v_n$ and $w_n'=\phi^{s_n}w_n$ be such that $\pi v'_n$ and $\pi w'_n$ are the projections of $p$ onto $c_{v_n}$ and $c_{w_n}$ respectively. So we have
\begin{equation*}
\begin{aligned}
&|t_n|=|b_{v_n^{-}}(\pi v_n',p)|\le d(\pi v_n',p)<\frac{\e}{10},\\
&|s_n|=|b_{w_n^{-}}(\pi w_n',p)|\le d(\pi w_n',p)<\frac{\e}{10}.
\end{aligned}
\end{equation*}
By passing to a subsequence, we can assume that $v_n'\to v, w_n'\to w$. Notice that as $\theta_n\to 0$, we know that $v^+=w^+=(v_0)^+$ and  $v^-=w^-=(v_0)^-$. Thus by Proposition \ref{crucial} $v=\phi^tv_0$ and $w=\phi^sv_0$ for some $t,s\in \RR$ with $0\le |s|,|t|\le \frac{\e}{10}$. So $d(\pi v, \pi w)\le \frac{\e}{5}$.

On the other hand, we know by the triangle inequality that
\begin{equation*}
\begin{aligned}
&d(\pi v,\pi w)=\lim_{n\to \infty}d(\pi v'_n,\pi w'_n)\\
\ge &\lim_{n\to \infty}(d(\pi v_n,\pi w_n)-d(\pi v_n,\pi v'_n)-d(\pi w_n,\pi w'_n))\\
\ge &\lim_{n\to \infty}(\frac{\e}{2}-|t_n|-|s_n|)\ge \frac{\e}{2}-\frac{\e}{10}-\frac{\e}{10}>\frac{\e}{5}.
\end{aligned}
\end{equation*}
We arrive at a contradiction. So $\text{diam\ }\pi H^{-1}(\bP\times \bF\times \{0\})\le \frac{\e}{2}$.

By the continuity of $H$, $H^{-1}(\bP\times \bF\times \{0\})\subset SX$ is closed. As the diameter of its projection under $\pi$ is no more than $\frac{\e}{2}$, we know $H^{-1}(\bP\times \bF\times \{0\})\subset SX$ is bounded. Thus it is compact.

Finally, by the triangle inequality, for any $0<\a\le \frac{\e}{2}$, we have $\text{diam\ } \pi B_\theta^\a <\frac{\e}{2}+2\a<4\e$.
\end{proof}

At last, we also assume that the choice of $\theta\in (0, \theta_1)$ satisfies the following properties. Since $\theta\mapsto \bar\mu(\bP_\theta\times \bF_\theta)$ is nondecreasing and hence has at most countably many discontinuities, we always choose $\theta$ to be the continuity point of this function, i.e.,
$$\lim_{\rho\to \theta}\bar\mu(\bP_\rho\times \bF_\rho)=\bar\mu(\bP_\theta\times \bF_\theta).$$
Furthermore, $\theta$ is chosen so that $\bar\mu(\partial\bP_\theta\times \partial\bF_\theta)=0.$ By the product structure of $B$ and $S$ and the definition of $m$, we have for any $\a\in (0,\frac{3\e}{2})$,
\begin{equation}\label{e:choice}
\begin{aligned}
\lim_{\rho\to \theta} m(S_\rho)=m(S_\theta), \quad \lim_{\rho\to \theta} m(B^\a_\rho)=m(B^\a_\theta),
\text{\ \ and\ \ }m(\partial B^\a_\theta)=0.
\end{aligned}
\end{equation}

The following is a direct corollary of Lemma \ref{equicon1}.
\begin{corollary}\label{equicon}
Given $v_0, p, \e>0$ as above, there exists $\theta_2>0$ such that for any $0<\theta<\theta_2$, if $\xi,\eta\in \bP_\theta$ and any $q$ lying within $4\e$ of $\pi H^{-1}(\bP_\theta\times \bF_\theta\times [0,\infty))$, we have $|b_\xi(q,p)-b_\eta(q,p)|<\e^2$. Similar result holds if the roles of $\bP_\theta$ and $\bF_\theta$ are reversed.
\end{corollary}
Let $\theta_0:=\min\{\theta_1, \theta_2\}$, where $\theta_1$ is given in Lemma \ref{diameter}, $\theta_2$ is given in Corollary \ref{equicon}. In the following, we always suppose that $0<\theta<\theta_0$ and \eqref{e:choice} is satisfied.

\section{Closing lemma}
The aim of this section is to establish a type of closing lemma \ref{closing}. We need study the properties of angle metric and prove a weak $\pi$-convergence theorem \ref{piconvergence}.

\subsection{Angle metric}
The \emph{angle metric} on $\partial X$ is defined as
$$\angle(\xi,\eta):=\sup_{x\in X}\angle_x(\xi,\eta), \quad \forall \xi, \eta\in \pX.$$
Then the angle metric defines a path metric $d_T$ on $\pX$, called the \emph{Tits metric}. Clearly, $d_T\ge \angle$.
The angle metric induces a topology on $\partial X$ finer than the visual topology.

Since $\angle$ is a supremum of a continuous function, it is lower semicontinuous:

\begin{lemma}(Cf. \cite[Proposition 2.7]{Wat})\label{lsc}
The angle metric is lower semicontinuous, i.e., if $\xi_n\to \xi$ and $\eta_n\to \eta$ in the visual topology, then
$$\angle(\xi,\eta)\leq \liminf_{n\to \infty}\angle(\xi_n,\eta_n).$$
\end{lemma}

Using Lemma \ref{nonincreasing} and the first variation formula, we get the following fact, see \cite[Lemma 1.1]{DF} for another proof.
\begin{lemma}(Cf. \cite[Lemma 2.9]{Wat} and \cite[Lemma 1.1]{DF})\label{angle}
Let $v\in SX$ point at $\xi\in \pX$, and let $\eta\in \pX$. Then
$$t\mapsto \angle_{c_v(t)}(\xi,\eta)$$
is a nondecreasing function. Equivalently,
$$\angle_{c_v(t_1)}(c(\infty),\eta)+\angle_{c_v(t_2)}(c(-\infty),\eta)\le \pi$$
for any $t_1<t_2$.
\end{lemma}

\begin{proposition}\label{limit}
Let $v\in SX$ point at $\xi\in \pX$, and let $\eta\in \pX$. Then
$$\angle(\xi,\eta)=\lim_{t\to\infty} \angle_{c_v(t)}(\xi,\eta).$$
\end{proposition}
\begin{proof}
The above limit exists by Lemma \ref{angle}. Let $q\in X$ and $t_n\to \infty$ as $n\to \infty$.
Denote by $v_n:=g^{t_n}v, p_n:=\pi v_n$. Let $c_n$ be the unique geodesic from $q$ to $p_n$. Write $v_n'=\dot c_n(d(q, p_n))$.

We claim that $\angle_{p_n}(v_n,v_n')\to 0$. Assuming the claim, then by Lemmas \ref{angle} and \ref{con},
\begin{equation*}
\begin{aligned}
&\lim_{n\to \infty}\angle_{p_n}(\eta,v_n)=\lim_{n\to \infty}\angle_{p_n}(\eta,v'_n)\\
\ge &\lim_{n\to \infty}\angle_q(\eta, \dot c_n(0))=\angle_q(\eta, \xi).
\end{aligned}
\end{equation*}
Since $q$ is arbitrary, the proposition follows.

It remains to prove the claim. Assume the contrary for contradiction, i.e., $\angle_{p_n}(v_n,v_n')\ge \a$ for some $\a>0$ by passing to a subsequence. Fix a fundamental domain $\F\subset X$ containing $\pi v$.
Let $\c_n\in \Gamma$ be such that $\c_n p_n\in \F$. Since $\overline \F$ is compact, by passing to a subsequence, we assume that $\c_n p_n\to p\in \overline \F$. Define $\sigma_n$ to be the geodesic from $\c_np_n$ to $\c_n \pi v$, and $\tau_n$ to be the geodesic from $\c_np_n$ to $\c_n q$. Denote $s_n:=d(\c_np_n, \c_n q)$ and recall that $t_n=d(\c_np_n, \c_n \pi v)$. By the triangle inequality, $|s_n-t_n|\le d(\c_n \pi v, \c_n q)=d(\pi v,q)$. So $\lim_{n\to \infty}s_n=\infty$. Without loss of generality, assume $s_n:=\min\{s_n,t_n\}$. For any $0\le t\le s_n$, by Lemma \ref{nonincreasing}(1),
$$d(\sigma_n(t), \tau_n(t))\le d(\sigma_n(s_n), \tau_n(s_n))\le 2d(\pi v,q).$$
Now by passing to a subsequence, suppose that $\dot\sigma_n(0)\to \dot\sigma(0)$ and $\dot\tau_n(0)\to \dot\tau(0)$. Then $\sigma$ and $\tau$ are two geodesics starting from $p$ with $\angle_p(\dot\sigma(0), \dot\tau(0))\ge\a>0$ and $d(\sigma(t), \tau(t))\le 2d(\pi v,q)$ for any $t>0$, which is a contradiction to Lemma \ref{nonincreasing}(1). This proves the claim and the proposition.
\end{proof}
\begin{remark}
If $X/ \Gamma$ is not compact, the above proposition holds if $v\in SX$ is assumed to be recurrent (cf. \cite[Corollary 2.11]{Wat}). Our proof above uses the compactness of $X/\Gamma$ and works for arbitrary $v\in SX$.
\end{remark}

The following lemma is a rigidity counterpart of Lemma \ref{angle}. It is a fundamental difference from the nonpositive curvature case that the geodesic considered in the lemma is required to be recurrent. We do not know if the recurrence assumption can be removed.

\begin{lemma}(\cite[Corollary 6.8]{Wat})\label{rigidity}
Let $c$ be a recurrent geodesic, and suppose there exists $\xi\in \pX$ such that
$\angle_{c(t)}(\xi, c(\infty))=\theta$ for all $t$ and some $0<\theta < \pi$. Then $c$ is the boundary of a flat half plane.
\end{lemma}
\begin{remark}\label{fact1}
In nonpositive curvature case, there is a much easier proof of the above lemma (without the recurrence assumption) using the following rigidity fact: If $t_1<t_2$ and $$\angle_{c(t_1)}(\xi, c(\infty))+\angle_{c(t_2)}(\xi, c(-\infty))=\pi,$$
then the infinite triangle with vertices $c(t_1), c(t_1), x$ is flat. This rigidity fact does not hold for general manifolds without foal points.
\end{remark}

The following proposition on Tits metric will not be used in subsequent discussions, but may have independent interest.

\begin{proposition}\label{pi}
Let $c$ be a recurrent geodesic, not the boundary of a flat half plane. Then
$$d_T(c(-\infty),c(\infty))>\pi.$$
\end{proposition}

\begin{proof}
At first, we have
$$d_T(c(-\infty),c(\infty))\ge \angle(c(-\infty),c(\infty))\ge \angle_{c(0)}(c(-\infty),c(\infty))=\pi.$$

Assume for contradiction that
$d_T(c(-\infty),c(\infty))=\pi.$ Recall that $(\pX, d_T)$ is a length space. It follows that there exist points $\eta_i\in \pX$ such that $\lim_{i\to \infty}d_T(\eta_i, c(-\infty))=\frac{\pi}{2}$ and $\lim_{i\to \infty}d_T(\eta_i, c(\infty))=\frac{\pi}{2}$. The sequence $\eta_i$ has a limit point $\eta\in \pX$ in the visual topology. By Lemma \ref{lsc}, we have
$\angle(\eta, c(-\infty))\le \frac{\pi}{2}$ and $\angle(\eta, c(-\infty))\le\frac{\pi}{2}$,
and hence by the triangle inequality,
$$\angle(\eta, c(-\infty))=\angle(\eta, c(-\infty))=\frac{\pi}{2}.$$
For any $t\in \RR$, we have
$$\angle_{c(t)}(\eta, c(-\infty))+\angle_{c(t)}(\eta, c(\infty))\ge \angle_{c(t)}(c(\infty), c(-\infty))=\pi.$$
On the other hand,

\begin{equation*}
\begin{aligned}
&\angle_{c(t)}(\eta, c(-\infty))\le \angle(\eta, c(-\infty))=\frac{\pi}{2},\\
&\angle_{c(t)}(\eta, c(\infty))\le \angle(\eta, c(\infty))=\frac{\pi}{2}.
\end{aligned}
\end{equation*}
It follows that $\angle_{c(t)}(\eta, c(\infty))=\frac{\pi}{2}$ for any $t\in \RR$. Applying Lemma \ref{rigidity}, we know that $c$ is the boundary of a half plane, a contradiction. This proves the lemma.
\end{proof}

\begin{remark}\label{fact3}
In nonpositive curvature case, we have $d_T(\xi,\eta)=\angle(\xi,\eta)$ if there is no geodesic connecting $\xi$ and $\eta$, and thus $$\angle(\xi,\eta)=\min\{d_T(\xi,\eta),\pi\}.$$
The proof uses comparison results due to the nonpositive curvature. We do not know how to prove the above fact for manifolds without focal points, since the comparison results were absent.
\end{remark}

\subsection{Weak $\pi$-convergence theorem}
The following is a weak form of $\pi$-convergence theorem for manifolds without focal points. It is not known if the $\pi$-convergence theorem \cite[Lemma 19]{PS} for CAT$(0)$ space can be extended to manifolds without focal points, since many properties of Tits metric (like the one discussed in Remark \ref{fact3}) are absent.
\begin{theorem}[Weak $\pi$-convergence theorem]\label{piconvergence}
Let $X$ be a simply connected manifold without focal points, $v_0, p, \e$ be fixed as in Section $2.4$, and $\theta_1$ be given in Proposition \ref{crucial}. Fix any $0<\rho<\theta<\theta_1$.

Suppose that $x\in X$, and $\c_i\in \Gamma$ such that $\c_i(x)\to p\in \bF_\rho$ and $\c^{-1}_i (x)\to n \in \bP_\rho$ as $i\to \infty$. Then for any open set $U$ with $U\supset \bF_\rho$, $\c_i(\bF_\theta) \subset U$ for all $i$ sufficiently large.
\end{theorem}

To prove the theorem, we need the following lemma. Denote $[x,n):=c_0([0,\infty))$ where $c_0$ is a geodesic ray from $x$ and pointing at $n$, and similarly $(p, x]:=c_1((-\infty,0])$ where $c_1$ is a geodesic ray from $x$ and pointing at $p$.
\begin{lemma}\label{pi2}
For any $\tau>0$, there exists a point $y\in [x,n)$ such that $\angle_y(n,c)>\pi-\tau$ for all $c\in \bF_\theta$.
\end{lemma}
\begin{proof}
Assume the contrary. Then there is a monotone sequence $\{y_i\}\subset [x,n)$ and a sequence of points $\{c_i\}\subset \bF_\theta$ such that $y_i\to n$ with $\angle_{y_i}(n,c_i)\le \pi-\tau$. By Lemma \ref{angle}, we know $\angle_{y_i}(n,c_k)\le \pi-\tau$ for all $k\ge i$. Since $\bF_\theta$ is compact, by passing to a subsequence, we assume that $c_k\to c\in \bF_\theta$. For any fixed $i$, the sequence of geodesic rays $[y_i,c_k)$ converge to $[y_i,c)$ as $n\to \infty$. Thus  $\angle_{y_i}(n,c_k)\to  \angle_{y_i}(n,c)\le \pi-\tau$. Since this is true for any $y_i\to n$, it follows from Proposition \ref{limit} that $\angle(n,c)\le \pi-\tau$. On the other hand, by Proposition \ref{crucial}, $c$ and $n$ are connected by a geodesic, and thus $\angle(n,c)\ge \pi$. A contradiction and the lemma is proved.
\end{proof}

\begin{proof}[Proof of Theorem \ref{piconvergence}]
Let $U\supset \bF_\rho$ be open in $\pX$. Fix $x\in X$. For any $\tau>0$, Lemma \ref{pi2} tells that there exists $y\in [x,n)$ such that  $\angle_y(n,c)>\pi-\tau$ for all $c\in \bF_\theta$.

Note that the geodesic segments $[x,\c_i^{-1}x]$ converge to the geodesic ray $[x,n)$. For each $i\in \NN$ let $y_i$ be the projection of $y$ on $[x,\c_i^{-1}x]$. Notice that $\c_i^{-1}x \to n$ as $i\to \infty$. So we may assume that $d(y_i, \c_i^{-1}(x))\ge 1$. Let $z_i\in [y_i,\c_i^{-1}x]$ such that $d(y_i,z_i)=1$. It is clear that $z_i\to z\in [y,n)$ with $d(y,z)=1$ (see \cite[p. 569]{PS}) for the figure).

By Lemma \ref{con} and uniform continuity, we have for any $c\in \bF_\theta$,
$$\angle_{y_i}(z_i,c)\to \angle_y(n,c)>\pi-\tau.$$
So for $i$ sufficiently large, $\angle_{y_i}(z_i,c)>\pi-\tau$ for any $c\in \bF_\theta$.

Now applying the action of $\c_i$ to the above geodesic segments, we have $\c_i([x,\c_i^{-1}x])=[\c_ix, x]\to (p,x]$ as $i\to \infty$. Moreover, for $i$ sufficiently large, $\angle_{\c_iy_i}(\c_iz_i,\c_ic)=\angle_{y_i}(z_i,c)>\pi-\tau$ for any $c\in \bF_\theta$.

Let $u\in (p, x]$, and for each $i$ let $u_i$ be the projection of $u$ to $[\c_i(x), x]$.
For $i$ sufficiently large, $u_i\in [\c_i(x), x]$ and $d(u_i, \c_iz_i)>1$. So there exists $v_i\in [\c_iz_i, u_i]$ such that $d(v_i, u_i)=1$. By Lemma \ref{angle},
$$\angle_{\c_iy_i}(\c_iz_i,\c_ic)+\angle_{u_i}(v_i,\c_ic)\le \pi$$
for all $c\in \bF_\theta$ and $i$ sufficiently large. It follows that
$$\angle_{u_i}(v_i,\c_ic)< \tau$$
for all $c\in \bF_\theta$ and $i$ sufficiently large.

Note that $v_i\to v$ for some $v\in (p,u]$ with $d(v,u)=1$. By Lemma \ref{con}, for $i$ sufficiently large,
$$\angle_{u}(p,\c_ic)=\angle_{u}(v,\c_ic)< \tau.$$
That is, for any $\tau>0$ and any $u\in (p,x]$, there is $i_0\in \NN$ such that for all $i\ge i_0$ and all $c\in \bF_\theta$, $\angle_{u}(p,\c_ic)< \tau.$

We claim that $\c_i\bF_\theta\subset U$ for all $i$ sufficiently large. Assume the contrary. Then by passing to a subsequence, for each $i$ there exists $c_i\in \bF_\theta$ such that $\c_ic_i\notin U$. Again by passing to a subsequence, we may assume that $\c_ic_i\to \hat c\in \pX\setminus U$. So $\angle(p, \hat c)>0$. Pick any $0<\tau<\angle(p, \hat c)$. By Lemma \ref{limit}, there exists $w\in (p,x]$ such that $\angle_w(p, \hat c)>\tau$. However, by the last paragraph, $\angle_{w}(p,\c_ic_i)< \tau$ for $i$ large enough. By continuity in Lemma \ref{con}, we arrive at a contradiction. This proves the claim and the theorem.
\end{proof}

\subsection{Closing lemma}
In the next section, we will count the number of elements in certain subsets of $\Gamma$. Let us collect the definitions here for convenience.

\begin{equation*}
\begin{aligned}
\C(t,\a)=\C_\theta(t,\a)&:=\{\c\in \C: S_\theta\cap \phi^{-t}\c_* B_\theta^\a\neq \emptyset\},\\
\C^*=\C^*_\theta&:=\{\c\in \C: \c\bF_\theta \subset \bF_\theta \text{\ and\ }\c^{-1}\bP_\theta\subset \bP_\theta\},\\
\C^*(t,\a)&:=\C^*\cap \C(t,\a),\\
\C'(t,\a)&:=\{\c\in \C^*(t,\a): \c\neq \b^n \text{\ for any\ } \b\in \C, n\ge 2\}.
\end{aligned}
\end{equation*}

\begin{lemma}[Closing lemma]\label{closing}
Let $X$ be a simply connected manifold without focal points, $v_0, p, \e$ be fixed as in Section $2.4$, and $\theta_1$ be given in Proposition \ref{crucial}. Then for every $0<\rho<\theta<\theta_1$, there exists some $t_0> 0$ such that for all $t\ge t_0$, we have $\C_\rho(t,\a)\subset \C_\theta^*$.
\end{lemma}
\begin{proof}
Let $U$ be the interior of $\bF_\theta$. Then $\bF_\rho\subset U$. We claim that there exists some $t_0> 0$ such that for all $t\ge t_0$ and $\c\in \Gamma$, if $S_\rho\cap \phi^{-t}\c B_\rho^\a\neq \emptyset$, then $\c (\bF_\theta)\subset U$.

Let us prove the claim. Assume not. Then for each $i$, there exist $t_i\to \infty$ and $\c_i\in \C$ such that $v_i\in S_\rho\cap \phi^{-t_i}\c_iB_\rho^\a$, but $\c_i(\bF_\theta) \nsubseteq U$. Clearly, for any $x\in X$, $\c_ix$ goes to infinity. By passing to a subsequence, let us assume that $\c_ix\to \xi\in \pX$.

Let $v_i\in S_\rho\cap \phi^{-t_i}\c_iB_\rho^\a$. By Lemma \ref{diameter}, $S_\rho$ and $B_\rho^\a$ are both compact. By passing to a subsequence, we may assume that $v_i\to v\in S_\rho$ and $\c_i^{-1}\phi^{t_i}v_i\to w\in B_\rho^\a$. Note that $\c_i \pi w\to \xi\in \pX$. Since $d(\c_iw, \phi^{t_i}v_i)\to 0$, we have $\xi=\lim_i \pi \phi^{t_i}v_i\in \bF_\rho$.

We may assume that $\c_i^{-1}\pi v\to \eta\in \pX$. Let $w_i=\c_i^{-1}\phi^{t_i}v_i\in B_\rho^\a$. Then $d(\c_i^{-1}v, \phi^{-t_i}w_i)=d(\c_i^{-1}v, \c_i^{-1}v_i)\to 0$, and thus $d(\c_i^{-1}\pi v, \pi \phi^{-t_i}w_i)\to 0$. We then see that $\eta=\lim_i \pi \phi^{-t_i}w_i\in \bP_\rho$.

Now we have $\c_i\pi v\to \xi\in \bF_\rho$ and $\c_i^{-1}\pi v\to \eta\in \bP_\rho$. Applying Theorem \ref{piconvergence}, we have $\c_i(\bF_\theta) \subset U$ for all $i$ sufficiently large. A contradiction and the claim follows.

By the claim, there exists some $t_0> 0$ such that for all $t\ge t_0$ and $\c\in \C_\rho(t,\a)$, we have $\c \bF_{\theta} \subset U\subset \bF_{\theta}$.

Analogously, by reversing the roles of $\bP$ and $\bF$, $S_\rho$ and $B_\rho^\a$, and the roles of $\c$ and $\c^{-1}$, we can prove that $\c^{-1} \bP_{\theta} \subset \bP_{\theta}$. Thus $\c\in \C_\theta^*$ and the proof of the lemma is completed.
\end{proof}

\section{Using scaling and mixing}
In this section, we use the scaling and mixing properties of Knieper measure $m$, to give an asymptotic estimates of $\#\C^*(t,\a)$ and $\#\C(t,\a)$.
\subsection{Depth of intersection}
To start, we want to show the relation between $t$ and $|\c|$ when $\c\in \C^*(t,\a)$.

Given $\xi\in \pX$ and $\c\in \C$, define $b_\xi^\c:=b_\xi(\c p, p)$.

\begin{lemma}\label{intersection1}
Let $\xi,\eta\in \bP$ and $c\in \C(t,\a)$ with $t>0$. Then $|b_\xi^\c-b_\eta^\c|<\e^2$.
\end{lemma}
\begin{proof}
The proof is an application of Corollary \ref{equicon}. The computation is completely parallel to that in \cite[Lemma 4.11]{CKW2}, and hence omitted here.
\end{proof}

\begin{lemma}\label{intersection2}
Let $c$ be an axis of $\c\in \C$ and $\xi=c(-\infty)$. Then $b_\xi^\c=|\c|$.
\end{lemma}
\begin{proof}
The proof involves computation of Busemann functions, which is completely parallel to that in \cite[Lemma 4.12]{CKW2}. Hence it is omitted here.
\end{proof}

\begin{lemma}\label{intersection3}
Given any $\c\in \C$ and any $t\in \RR$, we have
$$S\cap \phi^{-t}\c B^\a=\{w\in E^{-1}(\bP\times \c\bF): s(w)\in [0,\e^2]\cap (b_{w^{-}}^\c-t+[0,\a])\}.$$
\end{lemma}
\begin{proof}
The proof is completely parallel to that in \cite[Lemma 4.13]{CKW2}, and hence omitted here.
\end{proof}
\begin{lemma}\label{depth}
If $\c\in \C^*(t,\a)$, then $|\c|\in [t-\a-\e^2, t+2\e^2]$.
\end{lemma}
\begin{proof}
The proof is completely parallel to that in \cite[Lemma 4.14]{CKW2}, and hence omitted here. We just notice that the proof uses Lemmas \ref{intersection1}, \ref{intersection2} and \ref{intersection3} and also the fact that $\c\in \C^*$.
\end{proof}

The following lemma implies that the intersections also have product structure.
\begin{lemma}\label{depth1}
If $\c\in \C^*(t,\a)$, then
$$S\cap \phi^{-(t+2\e^2)}\c B^{\a+4\e^2}\supset H^{-1}(\bP\times \c\bF\times [0,\e^2]):=S^\c.$$
\end{lemma}
\begin{proof}
The proof is completely parallel to that in \cite[Lemma 5.1]{CKW2}. We reproduce the proof since it reflects that advantage of the use of the slice $S$.

Let $\c\in \C^*(t,\a)$, then $S\cap g^{-t}\c B^\a\neq \emptyset$. By Lemma \ref{intersection3}, there exists $\eta\in \bP$ such that
$$[0,\e^2]\cap (b_{\eta}^\c-t+[0,\a])\neq \emptyset.$$
It follows that $[0,\e^2]\subset (b_{\eta}^\c-t-\e^2+[0,\a+2\e^2]).$ Then by Lemma \ref{intersection1}, for any $\xi\in \bP$ we have
$$[0,\e^2]\cap (b_{\xi}^\c-t-\e^2+[0,\a+2\e^2])\neq \emptyset,$$
which in turn implies that
$$[0,\e^2]\subset  (b_{\xi}^\c-t-2\e^2+[0,\a+4\e^2]).$$
We are done by Lemma \ref{intersection3}.
\end{proof}
\subsection{Scaling and mixing calculation}
We use the following notations in the asymptotic estimates.
\begin{equation*}
\begin{aligned}
f(t)=e^{\pm C}g(t)&\Leftrightarrow e^{-C}g(t)\le f(t)\le e^{C}g(t) \text{\ for all\ } t;\\
f(t) \lesssim g(t) &\Leftrightarrow \limsup_{t\to \infty}\frac{f(t)}{g(t)}\le 1;\\
f(t) \gtrsim g(t) &\Leftrightarrow \liminf_{t\to \infty}\frac{f(t)}{g(t)}\ge 1;\\
f(t) \sim g(t) &\Leftrightarrow \lim_{t\to \infty}\frac{f(t)}{g(t)}= 1;\\
f(t)\sim e^{\pm C}g(t)&\Leftrightarrow e^{-C}g(t)\lesssim f(t)\lesssim e^{C}g(t).
\end{aligned}
\end{equation*}

\begin{lemma}\label{scaling1}
If $\c\in \C^*$, then
$$m(S^\c)=e^{\pm 2h\e}e^{-h|\c|}m(S).$$
\end{lemma}

\begin{proof}
The proof is completely parallel to that in \cite[Lemma 5.2]{CKW2}. The main work is to estimate $\b_p(\xi,\eta)$ and $b_\eta(\c^{-1}p,p)$ given $\xi\in \bP, \eta\in \bF$.

Firstly, take $q$ lying on the geodesic connecting $\xi$ and $\eta$ such that $b_\xi(q,p)=0$. Then
$$|\b_p(\xi,\eta)|=|b_\xi(q,p)+b_\eta(q,p)|=|b_\eta(q,p)|\le d(q,p)<\frac{\e}{2}$$
where we used Lemma \ref{diameter} in the last inequality.

Secondly, since $\c\in \C^*$, we know $\c^{-1}$ has a rank one axis $c$. By Lemma \ref{intersection2}, $b_{c(-\infty)}(\c^{-1}p,p)=|\c^{-1}|=|\c|$. Then by Corollary \ref{equicon}, we have $|b_\eta(\c^{-1}p,p)-|\c||<\e^2$ for any $\eta\in \bF$.

Notice that since $\c\bF\subset \bF$, we have $\c\eta\in \bF$ if $\eta\in \bF$. Thus we have
\begin{equation*}
\begin{aligned}
\frac{m(S^\c)}{m(S)}=&\frac{\e^2\bar \mu(\bP \times \c\bF)}{\e^2\bar \mu(\bP \times \bF)}=\frac{e^{\pm h\e/2}\mu_p(\bP)\mu_p(\c\bF)}{e^{\pm h\e/2}\mu_p(\bP)\mu_p(\bF)}\\
=&e^{\pm h\e}\frac{\mu_{\c^{-1}p}(\bF)}{\mu_p(\bF)}=e^{\pm h\e}\frac{\int_{\bF}e^{-hb_\eta(\c^{-1}p,p)}d\mu_p(\eta)}{\mu_p(\bF)}\\
=&e^{\pm h\e}\e^{\pm h\e^2}e^{-h|\c|}=e^{\pm 2h\e}e^{-h|\c|}.
\end{aligned}
\end{equation*}
\end{proof}

Combining Lemma \ref{depth} and Lemma \ref{scaling1}, we have
\begin{corollary}[Scaling]\label{scaling2}
Given $\a\le \frac{3}{2}\e$ and $\c\in \C^*(t,\a)$, we have $|t-|\c||\le 2\e$, and thus
$$m(S^\c)=e^{\pm 4h\e}e^{-ht}m(S).$$
\end{corollary}

\begin{remark}
It is clear that the conclusions in Lemma \ref{scaling1} and Corollary \ref{scaling2} hold if $m, S, S^\c$ are replaced by $\lm, \lS, {\lS}^\c$ respectively.
\end{remark}

Finally, we combine scaling and mixing properties of Knieper measure to obtain the following asymptotic estimates. The proof is a repetition of \cite[Section 5.2]{CKW2}. Since it is a key step, we provide a detailed proof here.
\begin{proposition}\label{asymptotic}
We have
\begin{equation*}
\begin{aligned}
e^{-4h\e}\lesssim &\frac{\#\C^*_\theta(t,\a)}{e^{ht}m(B_\theta^\a)}\lesssim e^{4h\e}(1+\frac{4\e^2}{\a}),\\
e^{-4h\e}\lesssim &\frac{\#\C_\theta(t,\a)}{e^{ht}m(B_\theta^\a)}\lesssim e^{4h\e}(1+\frac{4\e^2}{\a}).
\end{aligned}
\end{equation*}
\end{proposition}
\begin{proof}
Recall that $\a\in (0, \frac{3\e}{2}]$. By Lemmas \ref{closing} and \ref{depth1}, for any $0<\rho<\theta$ and $t$ large enough, we have
$$\lS_\rho\cap \phi^{-t}\lB_\rho^{\a}\subset \bigcup_{\c\in \C^*_\theta(t,\a)}\lS_\theta^\c \subset \lS_\theta \cap \phi^{-(t+2\e^2)}\lB_\theta^{\a+4\e^2}.$$
By Corollary \ref{scaling2}, $\lm(\lS_\theta^\c)=e^{\pm 4h\e}e^{-ht}\lm(\lS_\theta).$ Thus we have
\begin{equation*}
\begin{aligned}
e^{-4h\e}\lm(\lS_\rho\cap \phi^{-t}\lB_\rho^\a)&\le \#\C^*_\theta(t,\a)e^{-ht}\lm(\lS_\theta)\\
&\le e^{4h\e}\lm(\lS_\theta\cap \phi^{-(t+2\e^2)}\lB_\theta^{\a+4\e^2}).
\end{aligned}
\end{equation*}
Dividing by $\lm(\lS_\theta)\lm(\lB_\theta^\a)$ and using mixing of $\lm$, we get
\begin{equation}\label{e:sim}
\begin{aligned}
e^{-4h\e}\frac{m(S_\rho)m(B_\rho^\a)}{m(S_\theta)m(B_\theta^\a)} \lesssim \frac{\#\C^*_\theta(t,\a)}{e^{ht}m(B_\theta^\a)}
\lesssim e^{4h\e}\frac{m(B_\theta^{\a+4\e^2})}{m(B_\theta^\a)}.
\end{aligned}
\end{equation}
By \eqref{e:choice}, letting $\rho \nearrow \theta$, we obtain the first equation in the proposition.

To prove the second equation, we consider $\theta<\rho< \theta_0$. Then by Lemma \ref{closing}, $\C^*_\theta(t,\a)\subset \C_\theta(t,\a) \subset \C^*_\rho(t,\a).$ By \eqref{e:sim},
\begin{equation*}
\begin{aligned}
e^{-4h\e}\frac{m(S_\rho)m(B_\rho^\a)}{m(S_\theta)m(B_\theta^\a)} &\lesssim \frac{\#\C^*_\theta(t,\a)}{e^{ht}m(B_\theta^\a)}\lesssim \frac{\#\C_\theta(t,\a)}{e^{ht}m(B_\theta^\a)}\\
&\lesssim \frac{\#\C^*_\rho(t,\a)}{e^{ht}m(B_\theta^\a)}\lesssim e^{4h\e}\frac{m(B_\rho^{\a+4\e^2})}{m(B_\theta^\a)}.
\end{aligned}
\end{equation*}
Letting $\rho\searrow \theta$ and by \eqref{e:choice}, we get the second equation in the proposition.
\end{proof}

\section{Measuring along periodic orbits}
Recall that $C(t)$ is any maximal set of pairwise non-free-homotopic closed geodesics with length $(t-\e,t]$ in $M$. We obtain in this section an upper bound and a lower bound respectively for $\#C(t)$.
\subsection{Upper bound for $\#C(t)$}
By definition of $\nu_t$, we have
$$\#C(t)=\frac{\sum_{\lc\in C(t)} Leb_\lc(\lB_\theta^\a)}{t\nu_t(\lB_\theta^\a)}.$$
Define
$$\Pi(t):=\{\dot{\lc}(s)\in \pr H^{-1}(\bP\times \bF)\times \{0\}: \lc\in C(t), s\in \RR\}.$$
Then we have
\begin{equation}\label{e:ct}
\begin{aligned}
\#C(t)=\frac{\a\#\Pi(t)}{t\nu(B_\theta^\a)}.
\end{aligned}
\end{equation}

Now we define a map $\Theta: \Pi(t)\to \C(t,\e)$ as follows. Given $\lv\in \Pi(t)$, let $\ell=\ell(\lv)\in (t-\e,t]$ be such that $\phi^\ell\lv=\lv$. Let $v$ be the unique lift of $\lv$ such that $v\in H^{-1}(\bP\times \bF)\times \{0\}\subset B_\theta^\a$.
Define $\Theta(v)$ to be the unique axial isometry of $X$ such that $\phi^\ell v=\Theta(v) v$. Then $|\Theta(v)|=\ell$. If $\c=\Theta(v)$, then $\phi^tv=\phi^{t-\ell}\c v\in \c B_\theta^\e$. So $v\in S_\theta\cap \phi^{-t}\c B_\theta^\e$, and we get
\begin{equation}\label{e:theta}
\begin{aligned}
\Theta(\Pi(t))\subset \C(t,\e).
\end{aligned}
\end{equation}

To estimate $\#\Pi(t)$, we first show that $\Theta$ is injective.
\begin{lemma}\label{injective}
$\Theta$ is injective.
\end{lemma}
\begin{proof}
Suppose that $\lv, \lw\in \Pi(t)$ are such that $\Theta(\lv)=\Theta(\lw):=\c$.
Let $v,w\in B_\theta^\a$ be the lifts of $\lv, \lw$ respectively. Then by definition, both $c_v$ and $c_w$ are axes of $\c$.

We claim that $v^+=w^+$ and $v^-=w^-$. Indeed, first notice that $\c^n c_v(0)\to v^+$ and $\c^n c_w(0)\to w^+$. On the other hand,
$$d(\c^n c_v(0), \c^nc_w(0))=d(c_v(0), c_w(0))$$
for all $n\ge 0$. So $\c^n c_v(0)$ and $\c^n c_w(0)$ converge to the same point at infinity by the property of no focal points. Thus we get $v^+=w^+$. Similarly, $v^-=w^-$.

It follows that $c_v$ and $c_w$ are bi-asymptotic. If $c_v$ and $c_w$ are geometrically distinct, then they bound a flat strip by Lemma \ref{nonincreasing}(4). So $v$ and $w$ are singular vectors, which is a contradiction by the definition of $B_\theta^\a$ and Proposition \ref{crucial}. So $v$ and $w$ lie on a common geodesic. As $v,w\in H^{-1}(\bP\times \bF\times \{0\})$, we have $v=w$ and hence $\lv=\lw$. So $\Theta$ is injective.
\end{proof}

\begin{proposition}\label{upper}
We have
$$\#C(t)\le \frac{\e \#\C(t,\e)}{t\nu_t(\lB^\e)}.$$
\end{proposition}
\begin{proof}
The proposition follows from \eqref{e:ct}, \eqref{e:theta} and Lemma \ref{injective}.
\end{proof}
\subsection{Lower bound for $\#C(t)$}

First we deal with the multiplicity of $\c\in \C$. Given $\c\in \C$, let $d=d(\c)\in \NN$ be maximal such that $\c= \b^d$ for some $\b\in \C$. $\c\in \C$ is called \emph{primitive} if $d(\c)=1$, i.e., $\c\neq \b^d$ for any $\b\in \C$ and any $d\ge 2$.

Define $\C_2(\bP, \bF, t)$ to be the set of all $\c\in \C$ such that
\begin{enumerate}
  \item $\c$ has an axis $c$ with $c(-\infty)\in \bP, c(\infty)\in \bF$;
  \item $|\c|\in (t-\e,t]$;
  \item $d(\c)\ge 2$.
\end{enumerate}

\begin{lemma}\label{multi}
There exists $K>0$ such that for any $t>0$ we have
$$\sum_{\c\in \C_2(\bP, \bF, t)}d(\c)\le Ke^{\frac{2}{3}ht}.$$
\end{lemma}
\begin{proof}
Since $X$ is a simply connected manifold with no focal points, we have by \cite{FrMa}
$$h=\lim_{r\to \infty}\frac{1}{r}\log \text{Vol} (B(q,r))$$
where $B(q,r)$ is the ball of radius $r$ around $q$ in $X$. So the proof is almost analogous to that in \cite[Lemma 4.5]{CKW2} with only minor modifications as follows.

If $\c=\b(\c)^{d(\c)}$ for some $\b(\c)\in \C$, we must argue that $\b(\c)$ has an axis with endpoints in $\bP$ and $\bF$, so that we can choose $v\in H^{-1}(\bP\times \bF\times \{0\})$ tangent to such an axis with $\phi^{|\b(\c)|}v=\b(\c)v$. We cannot use \cite[Lemma 2.10]{CKW2} in our setting. Nevertheless, since $\c$ has a rank one axis with endpoints in $\bP$ and $\bF$, we know from the proof of Lemma \ref{injective} that this is the only axis for $\c$. Every axis of $\b(\c)$ is an axis of $\c$, so  $\b(\c)$ also has a unique axis, which is the axis of $\c$, as we want.

Repeat the remaining part of the proof of \cite[Lemma 4.5]{CKW2} and we are done.
\end{proof}

Recall that $\C'(t,\a):=\{\c\in \C^*(t,\a): \c\neq \b^n \text{\ for any\ } \b\in \C, n\ge 2\}$.

\begin{lemma}\label{lower}
Consider $\a=\e-4\e^2$. Then $\Theta(\Pi(t))\supset \C'(t-2\e^2, \a)$ and
$$\#C(t)\ge \frac{\a \#\C'(t-2\e^2,\a)}{t\nu_t(\lB^\a)}.$$
\end{lemma}
\begin{proof}
Let $\c\in \C'(t-2\e^2, \a)$.
Then there exists $v\in H^{-1}(\bP\times \bF\times \{0\})$ such that $\phi^{|\c|}v=\c v$.
By Lemma \ref{depth}, we have
\begin{equation*}
\begin{aligned}
|\c|&\ge (t-2\e^2)-\a-\e^2=t-\e+\e^2>t-\e,\\
|\c|&\le (t-2\e^2)+2\e^2=t.
\end{aligned}
\end{equation*}
It follows that $\underline c_{\lv}$ is a closed geodesic with length $|\c|\in (t-\e,t]$.
Note that if $\underline c$ is another closed geodesic in the free-homotopic class of $\underline c_{\lv}$, then we can lift $\underline c$ to a geodesic $c$ such that $c$ and $c_v$ are bi-asymptotic. So $c$ and $c_v$ bound a flat strip by Lemma \ref{nonincreasing}(4), which is a contradiction since $v$ is rank one. It follows that $\underline c_{\lv}$ is the only geodesic in its free-homotopic class. Thus $\underline c_{\lv}\in C(t)$.

As a consequence, $v\in \Pi(t)$ and $\c=\Theta(v)$. So $\Theta(\Pi(t))\supset \C'(t-2\e^2, \a)$ and thus by \eqref{e:ct},
$$\#C(t)\ge \frac{\a \#\C'(t-2\e^2,\a)}{t\nu_t(\lB^\a)}.$$
\end{proof}

\begin{proposition}\label{lower1}
Consider $\a=\e-4\e^2$. We have
$$\#C(t)\ge \frac{\a }{t\nu_t(\lB^\a)}\cdot (\#\C^*(t-2\e^2,\a)-Ke^{\frac{2}{3}ht}).$$
\end{proposition}
\begin{proof}
From the proof of Lemma \ref{lower}, we also see that $|\c|\in (t-\e,t]$ if $\c\in \C^*(t-2\e^2,\a)$. Thus
$$\C^*(t-2\e^2,\a)\setminus \C'(t-2\e^2,\a)\subset \C_2(\bP, \bF, t).$$
Then by Lemma \ref{multi},
$$\#\C^*(t-2\e^2,\a)- \#\C'(t-2\e^2,\a)\le Ke^{\frac{2}{3}ht}.$$
This together with Lemma \ref{lower} proves the proposition.
\end{proof}

\section{Equidistribution and completion of the proof}
The following result is standard in ergodic theory, which is a corollary of the classical proof of variational principle \cite[Theorem 9.10]{W}.
\begin{lemma}\cite[Proposition 4.3.12]{FH}\label{equilemma}
Let $Y$ be a compact metric space and $\phi$ a continuous flow on $Y$. Fix $\e > 0$ and suppose that $E_t\subset Y$ is a $(t,\e)$-separated set for all sufficiently large $t$. Define the measures $\mu_t$ by
$$\mu_t(A) :=\frac{1}{\#E_t}\sum_{v\in E_t}\frac{1}{t}\int_0^t\chi_A(\phi^sv)ds.$$
If $t_k\to \infty$ and the weak$^*$ limit $\mu=\lim_{k\to \infty}\mu_{t_k}$
exists, then
$$h_\mu(\phi^1)\ge \limsup_{k\to \infty}\frac{1}{t_k}\log \#E_{t_k}.$$
\end{lemma}

\begin{proof}[Proof of Theorem \ref{equi}]
We claim that the set $\{\underline{\dot  c}(0): c\in C(t)\}$ is $(t,\e)$-separated for any $0<\e<\inj(M)/2$.

Indeed, if it were not, then $C(t)$ would contain two closed geodesics $\underline c_1, \underline c_2$ in distinct free-homotopic classes
such that $d(\underline c_1(s), \underline c_2(s))\le \e< \inj(M)/2$ for all $s\in [0,t]$. Define $\lv=\underline{\dot c}_1(0)$ and $\lw=\underline{\dot c}_2(0)$.
We can lift  $\lv, \lw$ to $v,w\in SX$, and $\underline c_1, \underline c_2$ to $c_1,c_2$ respectively, such that $d(c_1(s), c_2(s))\le \e$ for all $s\in [0,t]$.
Moreover, there exist $\c_1,\c_2\in \C$ and $t_1,t_2\in (t-\e,t]$ such that $\c_1v=g^{t_1}v$ and $\c_2w=g^{t_2}w$. Then
\begin{equation*}
\begin{aligned}
d(\c_2^{-1}\c_1 c_1(0), c_2(0))=&d(\c_2^{-1}c_1(t_1), \c_2^{-1}c_2(t_2))=d(c_1(t_1), c_2(t_2))\\
\le &d(c_1(t_1), c_2(t_1))+|t_2-t_1|.
\end{aligned}
\end{equation*}
Hence $d(\c_2^{-1}\c_1 c_1(0), c_2(0))\le 3\e<2 \inj(M)$, which is possible only if $\c_1=\c_2$. Then $c_1$ and $c_2$ are both axes for a common $\c:=\c_1=\c_2$. By the proof of Lemma \ref{injective}, we see that $c_1$ and $c_2$ must be bi-asymptotic and consequently bound a flat strip. It follows that $\underline c_1$ and $\underline c_2$ are free-homotopic. A contradiction, so the claim holds.

Now by Propositions \ref{lower1} and \ref{asymptotic}, we know
\begin{equation*}
\begin{aligned}
\#C(t)\ge  &\frac{\a }{t\nu_t(\lB^\a)}\cdot (\#\C^*(t-2\e^2,\a)-Ke^{\frac{2}{3}ht})\\
\gtrsim  &\frac{\a }{t\nu_t(\lB^\a)}\cdot (e^{-4h\e}e^{ht}m(B_\theta^\a)-Ke^{\frac{2}{3}ht}).
\end{aligned}
\end{equation*}
So $\liminf_{t\to \infty}\frac{1}{t}\log \#C(t)\ge h.$ Applying Lemma \ref{equilemma}, we know any limit measure of $\nu_t$ has entropy equal to $h$, and thus it must be $m$, the unique MME. This proves Theorem \ref{equi}.
\end{proof}

\begin{proposition}\label{sumup}
We have
$$\#C(t)\sim e^{\pm Q\e}\frac{\e}{t}e^{ht}$$
where $Q>0$ is a universal constant depending only on $h$.
\end{proposition}
\begin{proof}
By Theorem \ref{equi} and \eqref{e:choice}, we have $\nu_t(\lB^\a)\to \lm(\lB^\a)=m(B^\a)$ for $\a=\e$ and $\a=\e-4\e^2$.

Consider $\a=\e$. By Propositions \ref{upper} and \ref{asymptotic}, we have
$$\#C(t)\lesssim \frac{\e \#\C(t,\e)}{tm(\lB^\e)} \lesssim e^{4h\e}(1+4\e)\frac{\e}{t}e^{ht}.$$

Now consider $\a=\e-4\e^2$. By Propositions \ref{lower1} and \ref{asymptotic},
\begin{equation*}
\begin{aligned}
\#C(t)\gtrsim &\frac{\a }{tm(\lB^\a)}\cdot (\#\C^*(t-2\e^2,\a)-Ke^{\frac{2}{3}ht})\\
\gtrsim &(1-4\e)e^{-4h\e}\frac{\e}{t}e^{-2h\e^2}e^{ht}.
\end{aligned}
\end{equation*}

As $0<\e<\frac{1}{8}$, there exists a universal $Q>0$ depending only on $h$ such that
$\#C(t)\sim e^{\pm Q\e}\frac{\e}{t}e^{ht}$.
\end{proof}

\begin{proof}[Proof of Theorem \ref{margulis}]
The last step of the proof of Theorem \ref{margulis} is to estimate $\#P(t)$ via $\#C(t)$ and a Riemannian sum argument.
Indeed, a verbatim repetition of the proof in \cite[Section 6.2]{CKW2} gives
$$\#P(t)\sim e^{\pm 2(Q+h)\e}\frac{e^{ht}}{ht}.$$
Since $\e>0$ can be arbitrarily small, we get $\#P(t)\sim \frac{e^{ht}}{ht}$ which completes the proof of Theorem \ref{margulis}.
\end{proof}

\ \
\\[-2mm]
\textbf{Acknowledgement.}
This work is supported by NSFC Nos. 12071474 and 11701559.

\section{Appendix: Bernoulli property of Knieper measure}

We prove Theorem \ref{bernoulli} in this appendix. To start, we recall some necessary definitions.
\begin{definition}[Kolmogorov property]
Let $T: X\to X$ be an invertible measure-preserving transformation in a Lebesgue space $(X,\mathcal{B},\mu)$. We say that $T$ has the \emph{Kolmogorov property}, or that $T$ is \emph{Kolmogorov}, if there is a sub-$\sigma$-algebra $\mathcal{K}$ of $\mathcal{B}$ which satisfies
\begin{enumerate}
  \item $\mathcal{K}\subset T\mathcal{K}$,
  \item  $\bigvee_{i=0}^{\infty}T^{i}\mathcal{K}=\mathcal{B}$,
  \item $\bigcap_{i=0}^{\infty}T^{-i}\mathcal{K}=\{\emptyset,X\}$.
\end{enumerate}
\end{definition}

Consider $k\in \NN$ and a probability vector $p=(p_0,\cdots, p_{k-1})$, i.e., $p_i\ge 0, i=0,\cdots, k-1$ and $\sum_{i=0}^{k-1}p_i=1$. A \emph{Bernoulli shift} is a transformation $\sigma$ on the space $\Sigma_k:=\{0,\cdots,k-1\}^{\ZZ}$ given by
$\sigma((x_n)_{n\in \ZZ})=(x_{n+1})_{n\in \ZZ}$ which preserves the Bernoulli measure $p^{\ZZ}$ on $\Sigma_k$.

\begin{definition}[Bernoulli property]
Let $T: X\to X$ be an invertible measure-preserving transformation in a Lebesgue space $(X,\mathcal{B},\mu)$. We say that $T$ has the \emph{Bernoulli property} or is called \emph{Bernoulli} if it is measurably isomorphic to a Bernoulli shift.
\end{definition}
Obviously, the Bernoulli property implies the Kolmogorov property.
\begin{definition}
A measure-preserving flow $(X,\phi,\mu)$ where $\phi=(\phi^t)_{t\in \RR}$ is said to have the \emph{Kolmogorov/Bernoulli property} if for every $t\neq 0$, the invertible measure-preserving transformation $(X, \phi^t, \mu)$ has the Kolmogorov/Bernoulli property.
\end{definition}

For the the geodesic flow on manifolds of nonpositive curvature, Call and Thompson \cite{CT} showed that the unique equilibrium state for certain potential has the Kolmogorov property, based on the earlier work of \cite{Led}. Recently, this result is extended to manifolds without focal points in \cite{CKP2}, via methods developed in \cite{CT}. In particular, we have
\begin{proposition}\cite[Theorem B]{CKP2}
Let $M$ be a rank one closed Riemannian manifold without focal points, then
the Knieper measure $m$ is Kolmogorov.
\end{proposition}

A classical argument in \cite{OW1} showed that the Kolmogorov property implies Bernoulli property for smooth invariant measures of Anosov flows. The argument was also carried out by Chernov and Haskell \cite{CH} for smooth invariant measures of suspension flows over some nonuniformly hyperbolic maps with singularities. It also works for hyperbolic invariant measures with local product structure. See also \cite{PV, PTV} for an introduction to terminology and ideas of the Kolmogorov-Bernoulli equivalence.

Based on the argument in Chernov and Haskell \cite{CH}, Call and Thompson \cite[Section $7$]{CT} proved that the Knieper measure $m$ is Bernoulli when $M$ has nonpositive curvature. Following \cite{CT}, we show that the Knieper measure $m$ is Bernoulli for manifolds without focal points. The main step is to show the existence of $\e$-regular coverings as follows.

\begin{definition}
A \emph{rectangle} in $SM$ is a measurable set $R\subset SM$, equipped with a distinguished point $z \in R$ with the property that for all points $x, y \in R$ the local weak stable manifold $W_\loc^{0s}(x)$ and the local unstable $W_\loc^{u}(y)$ intersect each other at a single point, denoted by $[x,y]$, which lies in $R$.
\end{definition}
Notice that a rectangle $R$ can be thought of as the Cartesian product of $W_\loc^{0s}(z)\cap R$ and $W_\loc^{u}(z)\cap R$, where a point $y\in R$ is given by $[W_\loc^{0s}(y)\cap W_\loc^{u}(z), W_\loc^{u}(y) \cap W_\loc^{0s}(z)]$.

Given a probability measure $\mu$ on $SM$, there is a natural product measure
$$\mu^p_R:=\mu^u_z\times \tilde\mu^{0s}_z,$$
where $\mu^u_z$ is the conditional measure induced by $\mu$ on $W_\loc^u(z)\cap R$ with respect to the measurable partition of $R$ into local unstable manifolds, and $\tilde\mu^{0s}_z$ is the corresponding factor measure on $W^{0s}_\loc(z)$.

Below is Chernov and Haskell's definition of $\e$-regular coverings.
\begin{definition}\label{coveringdef}
Given any $\e> 0$, we define an \emph{$\e$-regular covering} for a probability measure $\mu$ of $SM$ to be a finite collection of disjoint rectangles $\mathcal{R}=\mathcal{R}_\e$ such that
\begin{enumerate}
  \item $\mu(\cup_{R\in \mathcal{R}}R) > 1-\e$;
  \item Given any two points $x, y\in R \in \mathcal{R}$, which lie in the same local unstable or local weak stable manifold, there is a smooth curve on that manifold which connects $x$
and $y$ and has length less than $100\cdot \text{diam}R$;
  \item For every $R\in \mathcal{R}$ with a distinguished point $z\in R$, the product measure $\mu^p_R:=\mu^u_z\times \tilde\mu^{0s}_z$ satisfies
$$\left|\frac{\mu^p_R(R)}{\mu(R)}-1\right| <\e.$$
Moreover, $R$ contains a subset $G$ with $\mu(G) > (1-\e)\mu(R)$ such that for all $x\in G$,
$$\left|\frac{d\mu^p_R}{d\mu}(x)-1\right| <\e.$$
\end{enumerate}
\end{definition}

\begin{lemma}\label{covering}
For any $\d>0$ and $\e>0$, there exists an $\e$-regular covering of connected rectangles  $\mathcal{R}_{\e}$ for Knieper measure $m$ of $SM$, with $\diam(R)<\d$ for any $R\in \mathcal{R}_{\e}$.
\end{lemma}

\begin{proof}
Because $m(\text{Reg}) = 1$, $m$ has no zero Lyapunov exponent except in the flow direction. Thus $m$ is a hyperbolic measure. By \cite[Lemma 8.3]{Pe1}, \cite[Lemma 1.8]{Pe2} and \cite[Lemma 9.5.7]{BP}, for a hyperbolic measure $m$, we can find a finite
collection of disjoint rectangles $R$ covering a Pesin set for the Lyapunov regular
points for $m$. Choose a Pesin set with $m$-measure at
least $1-\e$ gives the first condition in Definition \ref{coveringdef}. Moreover, the rectangles $R$ can be chosen such that $\diam R<\d$. Since the metrics on local unstable or local weak stable manifold are uniformly equivalent to the Riemannian distance if $\d$ is small enough, the second condition in Definition \ref{coveringdef} is also satisfied.

To verify the third condition in Definition \ref{coveringdef}, we use the local production structure of $m$ provided by \eqref{mme1} and \eqref{mme2}. If $v\in \text{Reg}$, $P^{-1}(v^-, v^+)$ consists of the unit vectors tangent to geodesic $c_v$, and thus in the term $\text{Vol}\{\pi(P^{-1}(\xi,\eta)\cap A)$ in \eqref{mme2}, $\text{Vol}$ just becomes the length along the geodesic $c_v$. Since $m(\text{Reg}) = 1$, we in fact have for $m \ae v$,
$$dm(v)=f(\xi, \eta)d\mu_p(\xi)d\mu_p(\eta)dt,$$
where $P(v)=(\xi,\eta)$ and the density function $f(\xi,\eta):=e^{h\cdot \beta_{p}(\xi,\eta)}$. By Corollary \ref{continuous}, $f$ is continuous.

Let $R$ be a rectangle of sufficiently small diameter constructed above. We lift all objects to the universal cover $X$. Take a lift of $R$, which is still denoted by $R$. Since the local weak stable
and local unstable manifolds at $v$ intersect transversally if and only $v\in \text{Reg}$, it follows
that $R\subset \text{Reg}$. Notice that for $x\in \text{Reg}$, there exists a continuous map $\varphi_x: W_\loc^u(x)\to \pX$ given by $\varphi_x(v)=v^+$ where $v\in W_\loc^u(x)$. The conditional measure $\mu^u_x$ on $R\cap W_\loc^u(x)$ is given by
\begin{equation}\label{conditional}
d\mu^u_x(v) = \frac{f(x^-,\phi_xv) d\mu_p(\phi_xv)}{\int_{R\cap W_\loc^u(x)}f(x^-,\phi_xv) d\mu_p(\phi_xv)}
\end{equation}
and $\int d m(v)=c\int d\mu^u_x(v) d\mu_p(x^-)dt$ for some normalization constant $c>0$.

Given two points $x, y\in R$, the local weak stable holonomy map $\pi^{0s}_{xy}: W_\loc^{u}(x)\cap R \to W_\loc^{u}(y)\cap R$ is defined by
$$\pi_{xy}^{0s}(w):=[w,y]=W_\loc^{u}(y)\cap W_\loc^{0s}(w), \quad w\in W_\loc^{u}(x).$$
Note that $\phi_x(w)=\phi_{y}(\pi^{0s}_{xy}w):=\eta_w$. By \eqref{conditional}, the Jacobian of the holonomy map \begin{equation*}
\begin{aligned}
\left|\frac{d(\pi_{yx}^{0s})_*\mu^u_y}{d\mu^u_x}(w)\right|=\left|\frac{f(y^-,\eta_w)}{f(x^-,\eta_w)}\cdot\frac{\int_{R\cap W_\loc^u(x)}f(x^-,\eta_w) d\mu_p(\eta_w)}{\int_{R\cap W_\loc^u(x)}f(y^-,\eta_w) d\mu_p(\eta_w)}\right|.
\end{aligned}
\end{equation*}
Since $f$ is uniformly continuous and bounded away from $0$, by taking $\text{diam\ }R<\d$ small enough, we have that
$\left|\frac{d(\pi_{yx}^{0s})_*\mu^u_y}{d\mu^u_x}(w)-1\right| <\e$ for $x,y\in R$.
By definition of conditional measures (see Lemma 2.12 in \cite{CPZ}), this implies the third condition in Definition \ref{coveringdef}.
\end{proof}

\begin{proof}[Proof of Theorem \ref{bernoulli}]
As in \cite{CH}, if a smooth invariant measure or a hyperbolic measure $\mu$ is Kolmogorov and there exists an $\e$-regular covering
with non-atomic conditionals for $\mu$ for any $\e > 0$, then any finite partition $\xi$ of the
phase space with piecewise smooth boundary and a constant $C > 0$ such that
$\mu(B(\partial \xi, \d)) \le C\d$ for all $\d > 0$ is \emph{Very Weak Bernoulli} (see e.g. \cite{PTV, PV} for definition). A refining sequence of such
partitions with diameter going to $0$ suffices to conclude the Bernoulli property for
$\mu$. Such a sequence of partitions exist in this setting by \cite[Lemma 4.1]{OW2}.

As for Knieper measure $m$, we have showed in Lemma \ref{covering} that $\e$-regular coverings for $m$ exist for all $\e > 0$. Notice that the conditional measure of $m$ is non-atomic, since $m$ is a hyperbolic ergodic measure with positive entropy. We conclude that $m$ is Bernoulli.
\end{proof}


\begin{thebibliography}{10}
\bibitem{An}
D.~V.~Anosov, \emph{Geodesic flows on closed Riemannian manifolds with negative curvature}, Proc. Steklov Inst. Math. 90, (1967), 1-235.

\bibitem{Ba}
 M. Babillot, \emph{On the mixing property for hyperbolic systems}, Israel Journal of Mathematics, 2002, 129(1): 61-76.



\bibitem{BP}
L.~Barreira and Ya B.~Pesin, \emph{Nonuniform hyperbolicity: Dynamics of systems with nonzero Lyapunov exponents}, Encyclopedia of Mathematics and its Applications, 115, Cambridge University Press, Cambridge, 2007.

\bibitem{Bo1} R. Bowen,
 \emph{Periodic orbits for hyperbolic flows}, Amer. J. Math., 94, (1972), 1-30.


\bibitem{Bo2} R. Bowen,
   \emph{Maximizing entropy for a hyperbolic flow}, Math. Systems Theory, 7, (1974), 300-303.

\bibitem{BCFT}
K. Burns, V. Climenhaga, T. Fisher and D. J. Thompson, \emph{Unique equilibrium states for geodesic flows in nonpositive curvature}, Geom. Funct. Anal. 28, (2018), 1209-1259.


\bibitem{BuKa}
K. Burns and A. Katok, \emph {Manifolds with non-positive curvature,} Ergodic Theory Dynam. Systems, 5, (1985), 307-317.

\bibitem{CT}
B. Call and D. J. Thompson, \emph{Equilibrium states for products of flows and the mixing properties of rank 1 geodesic flows}, arXiv preprint arXiv:1906.09315, 2019.


\bibitem{CKP1} D. Chen, L. Kao and K. Park, \emph{Unique equilibrium states for geodesic flows over surfaces without focal points}, Nonlinearity, 33 (2020), no. 3, 1118-1155.

\bibitem{CKP2} D. Chen, L. Kao and K. Park, \emph{Properties of equilibrium states for geodesic flows over manifolds without focal points}, Advances in Mathematics, 2021, 380: 107564.


\bibitem{CH}
N. I. Chernov and C. Haskell, \emph{Nonuniformly hyperbolic K-systems are Bernoulli}, Ergodic
Theory Dynam. Systems 16 (1996), no. 1, 19-44.

\bibitem{CKW1} V. Climenhaga, G. Knieper and K. War, \emph{Uniqueness of the measure of maximal entropy for geodesic flows on certain manifolds without conjugate points}, Advances in Mathematics, 376 (2021), 107452, 44 pp.

\bibitem{CKW2}
V. Climenhaga, G. Knieper and K. War, \emph{Closed geodesics on surfaces without conjugate points}, arXiv:2008.02249, 2020.

\bibitem{CPZ}
V. Climenhaga, Ya. B. Pesin and A. Zelerowicz, \emph{Equilibrium measures for some partially hyperbolic systems}, Journal of Modern Dynamics, 16 (2020), 155-205.

\bibitem{CS}
C. B. Croke and V. Schroeder, \emph {The fundamental group of compact manifolds without conjugate points,} Comment. Math. Helv., 61, (1986), 161-175.

\bibitem{DPS}
F. Dal'Bo, M. Peign\'{e} and A. Sambusetti, \emph{On the horoboundary and the geometry of rays of negatively curved manifolds}, Pacific J. Math., 259, (2012), 55-100.

\bibitem{DF}
M. J. Druetta and C. J. Ferraris, \emph{The duality condition on manifolds without focal points}, Archiv der Mathematik, 1981, 36(1): 370-377.


\bibitem{Eb1}
P.~Eberlein, \emph{Geodesic flow in certain manifolds without conjugate points}, Trans. Amer. Math. Soc. 167 (1972), 151-170.

\bibitem{EO}
P.~Eberlein and B.~O'Neill, \emph{Visibility manifolds}, Pacific Journal of Mathematics, 46, no. 1 (1973), 45-109.

\bibitem{FH}
T. Fisher and B. Hasselblatt, \emph{Hyperbolic flows}. 2019, ISBN online 978-3-03719-700-4
DOI 10.4171/200

\bibitem{FrMa} A. Freire and R. Ma\~{n}\'{e},
\emph{On the entropy of the geodesic flow in manifolds without conjugate points},
Invent. Math., 69, (1982), 375-392.

\bibitem{GR} K. Gelfert and R. Ruggiero, \emph{Geodesic flows modelled by expansive flows}, Proc. Edinb. Math. Soc. (2), 62 (2019), 61-95.

\bibitem{Gun}
R. Gunesch, \emph{Counting closed geodesics on rank one manifolds}, preprint, arXiv: 0706.2845, 2007.


\bibitem{KH}
A.~Katok and B.~Hasselblatt, \emph{Introduction to the modern theory of dynamical systems}, Vol. 54. Cambridge university press, 1997.

\bibitem{Kn1}
 G. Knieper, \emph{On the asymptotic geometry of nonpositively curved manifolds}, Geom. Funct. Anal. 7, (1997), 755-782.

\bibitem{Kn2}
G.~Knieper, \emph{The uniqueness of the measure of maximal entropy for geodesic flows on rank 1 manifolds}, Ann. of Math. 148, (1998), 291-314.

\bibitem{Kn3}
G.~Knieper, \emph{Hyperbolic dynamics and Riemannian geometry}, Handbook of
dynamical systems, Vol. 1A, North-Holland, Amsterdam, 2002, pp. 453-545.

\bibitem{Led}
F. Ledrappier, \emph{Mesures d'\'{e}quilibre d'entropie compl\`{e}tement positive}, Ast\'{e}risque 50 (1977), 251-272.

\bibitem{LLW}
F. Liu, X. Liu, F. Wang, \emph{On the mixing and Bernoulli properties for geodesic flows on rank 1 manifolds without focal points}, Discrete and Continuous Dynamical Systems. Series A., to appear.

\bibitem{LWW}
F. Liu, F. Wang and W. Wu, \emph{On the Patterson-Sullivan measure for geodesic flows on rank 1 manifolds without focal points}, Discrete and Continuous Dynamical Systems. Series A. 40, no.3 (2020), 1517-1554.


\bibitem{Mar1}
G. A. Margulis, \emph{Certain applications of ergodic theory to the investigation
of manifolds of negative curvature}, Funkcional. Anal. i Prilo\v{z}en. 3 (1969), no. 4, 89-90.

\bibitem{Mar2}
G. A. Margulis, \emph{On some aspects of the theory of Anosov systems},
Springer Monographs in Mathematics, Springer-Verlag, Berlin, 2004,
With a survey by Richard Sharp: Periodic orbits of hyperbolic flows,
Translated from the Russian by Valentina Vladimirovna Szulikowska.


\bibitem{OW1}
D. Ornstein and B. Weiss, \emph{Geodesic flows are Bernoullian}, Israel J. Math., 14 (1973), 184-198.

\bibitem{OW2}
D. Ornstein and B. Weiss, \emph{On the Bernoulli nature of systems with some hyperbolic
structure}, Ergodic Theory Dynam. Systems 18 (1998), no. 2, 441-456.

\bibitem{OS}
J.~J.~O'Sullivan, \emph{Riemannian manifolds without focal points}, Journal of Differential Geometry, 11, (1976), 321-333.

\bibitem{PS}
P. Papasoglu and E. Swenson, \emph{Boundaries and JSJ decompositions of CAT$(0)$-groups}, Geom. Funct. Anal. 19 (2009), no. 2, 559-590.

\bibitem{PP}
W. Parry and M. Pollicott, \emph{An analogue of the prime number theorem for closed orbits of Axiom A flows}, Ann. of Math. (2) 118
(1983), no. 3, 573-591.

\bibitem{Pat} S. J. Patterson, \emph{The limit set of a Fuchsian group}, Acta Math., 136, (1976), 241-273.

\bibitem{Pe1}
Ja. B. Pesin, \emph{Characteristic Ljapunov exponents, and smooth ergodic theory}, Uspehi Mat.
Nauk 32 (1977), no. 4 (196), 55-114.

\bibitem{Pe2}
Ja. B. Pesin, \emph{Geodesic flows in closed Riemannian manifolds without focal points}, Izv. Akad. Nauk
SSSR Ser. Mat. 41 (1977), no. 6, 1195-1228.

\bibitem{PTV}
G. Ponce, A. Tahzibi and R. Var\~{a}o, \emph{On the Bernoulli property for certain partially hyperbolic diffeomorphisms}, Advances in Mathematics, 2018, 329: 329-360.

\bibitem{PV}
G. Ponce and R. Var\~{a}o, \emph{An Introduction to the Kolmogorov-Bernoulli Equivalence}, Springer International Publishing, 2019.

\bibitem{Rat}
M. Ratner, \emph{Anosov flows with Gibbs measures are also Bernoullian}, Israel Journal of Mathematics, 1974, 17(4): 380-391.

\bibitem{Ri}
R. Ricks, \emph{Counting closed geodesics in a compact rank one locally CAT($0$) space}, arXiv preprint (2019), arXiv:1903.07635


\bibitem{Ru1}
R. Ruggiero, \emph{Expansive geodesic flows in manifolds with no conjugate points,} Ergodic Theory Dynam. Systems, 17, (1997), 211-225.

\bibitem{Ru2}
R. Ruggiero, \emph{Dynamics and global geometry of manifolds without conjugate points,} Ensaios Matematicos, 12, (2007), 1-181.

\bibitem{Sul} D. Sullivan, \emph{The density at infinity of a discrete group of hyperbolic motions}, Publ. Math. IHES., 50, (1979), 171-202.

\bibitem{W}
P. Walters, \emph{An introduction to ergodic theory}, Graduate Texts in Mathematics, 79, Springer Science \& Business Media, 2000.

\bibitem{Wat}
J.~Watkins, \emph{The higher rank rigidity theorem for manifolds with no focal points},
Geom. Dedicata, 164, (2013), 319-349.

\bibitem{We}
B. Weaver, \emph{Growth rate of periodic orbits for geodesic flows over
surfaces with radially symmetric focusing caps}, J. Mod. Dyn. 8 (2014), no. 2, 139-176.

\end{thebibliography}
\end{document}